\documentclass[a4paper,11pt]{amsart}

\usepackage[utf8]{inputenc}
\usepackage[dvips]{graphicx}
\usepackage[shortlabels]{enumitem}
\usepackage{amsmath,amssymb,color,enumitem,graphics,hyperref,latexsym,pgfplots,tikz,vmargin}
\usepackage{relsize}
\usepackage{dsfont}

\hypersetup{
unicode=false,          % non-Latin characters in Acrobat’s bookmarks
pdftoolbar=true,        % show Acrobat’s toolbar?
pdfmenubar=true,        % show Acrobat’s menu?
pdffitwindow=false,     % window fit to page when opened
pdfstartview={FitH},    % fits the width of the page to the window
pdftitle={Sets with large intersection properties in metric spaces},    % title
pdfauthor={Felipe Negreira, Emiliano Sequeira},     % author
pdfsubject={11J83, 28A78},   % subject of the document
pdfkeywords={Diophantine approximations, metric spaces, net measures}, % list of keywords
pdfnewwindow=true,      % links in new PDF window
colorlinks=true,       % false: boxed links; true: colored links
linkcolor=black,          % color of internal links (change box color with linkbordercolor)
citecolor=black,        % color of links to bibliography
filecolor=magenta,      % color of file links
urlcolor=cyan           % color of external links
}

\def\diam{{\mathrm{diam}}}

\def\Hh{{\mathcal{H}}}

\def\N{{\mathbb{N}}}

\def\R{{\mathbb{R}}}

\def\H{{\mathbb{H}}}

\newcommand{\supp}{\operatorname{supp}}
\newcommand{\Vol}{\operatorname{Vol}}
\newcommand{\Ker}{\operatorname{Ker}}
\newcommand{\I}{\operatorname{Im}}

\newtheorem{lemma}{Lemma}[section]
\newtheorem{proposition}[lemma]{Proposition}
\newtheorem{theorem}[lemma]{Theorem}

\theoremstyle{definition}

\theoremstyle{remark}
\newtheorem{remark}[lemma]{Remark}

\begin{document}

\title{Relative $L^p$-cohomology and Heintze groups}
\author[E. Sequeira]{Emiliano Sequeira}
\address{Universidad de la República; Montevideo, Uruguay.}
\email{esequeira@cmat.edu.uy}


\begin{abstract}
We introduce the notion of \textit{relative $L^p$-cohomology} as a quasi-isometry invariant defined for Gromov-hyperbolic spaces, and apply it to the problem of quasi-isometry classification of Heintze groups. More precisely, we explicitly construct non-zero relative $L^p$-cohomology classes on a Heintze group of the form $\R^{n-1}\rtimes_\alpha\R$, which gives a way to prove that the eigenvalues of $\alpha$, up to a scalar multiple, are invariant by quasi-isometries.

In the case of degree $1$ we show a relation between the relative and the classical $L^p$-cohomology.
\end{abstract}

\maketitle

\section{Introduction}

In the context of some classes of metric spaces, $L^p$-cohomology is a quasi-isometry invariant with interesting applications to classification problems. In particular, it has been used to distinguish, up to quasi-isometries, connected homogeneous Riemannian manifolds of negative curvature. 
These manifolds are, due to a result by Heintze \cite{Heintze74}, characterized by the Lie groups of the form $G=N\rtimes_\alpha\R$ equipped with left-invariant metrics, where $N$ is a connected and simply connected nilpotent Lie group and $\alpha$ is a derivation on the Lie algebra of $N$ whose eigenvalues all have positive real part. They are the so called \textit{Heintze groups}.

In the (more general) case of a Riemannian manifold $M$, its $L^p$-cohomology (for $p\geq 1$) is a family of vector topological spaces $L^p H^k(M)$ (with  $k\in\N$)  defined similarly to de Rham cohomolgy but imposing a $L^p$-integrability condition to the forms. The invariance then occurs in the following sense: if $M$ and $N$ are quasi-isometric, then $L^pH^k(M)$ is isomorphic to $L^pH^k(N)$ for every $p$ and $k$. A discrete version of $L^p$-cohomology is also defined considering $\ell^p$-summable cochains on a simplicial complex.

The following result by Pansu gives a way to distinguish two Heintze groups with a particular structure. Its proof uses $L^p$-cohomology. 

\begin{theorem}[\cite{Pa08}]\label{ResultadoPrincipal}
Let $G_1=\R^{n-1}\rtimes_{\alpha_1}\R$ and $G_2=\R^{n-1}\rtimes_{\alpha_2}\R$ be two purely real Heintze groups, i.e. $\alpha_1$ and $\alpha_2$ only have real eigenvalues. If $G_1$ and $G_2$ are quasi-isometric, then there exists $\lambda>0$ such that $\alpha_1$ and $\lambda\alpha_2$ have the same eigenvalues counted with multiplicity. 
\end{theorem}

A more general version of Theorem \ref{ResultadoPrincipal} is proved in \cite{CS} using all previously known machinery \cite{C,LX,Pa89a,Xie14}. There exists a complete quasi-isometry characterization of purely real Heinze groups of the form $\R^{n-1}\rtimes_\alpha\R$, which also implies Theorem \ref{ResultadoPrincipal} (see \cite{Xie14}).

The strategy used by Pansu for proving Theorem \ref{ResultadoPrincipal} is to show that $L^pH^k(G_i)$ (with $k\geq 2$ and $i=1,2$) is zero for $p$ bigger than a critical exponent $p_k(i)$, and that it is not zero if $p$ belongs to an interval of the form $(p_k(i)-\epsilon,p_k(i))$. The quasi-isometry between $G_1$ and $G_2$ implies $p_k(1)=p_k(2)$ for every $k$. Then the result is obtained using the relation between critical exponents and the eigenvalues of the derivations. The non-vanishing of the $L^p$-cohomology involves the explicit construction of non-zero elements, which is a difficult problem from a technical point of view. The goal of this work is to show a shortcut at this point, which consists of slightly modifying the definition of $L^p$-cohomology in order to obtain a family of topological vector spaces with which it is easier to work. In addition, we do it in a self-contained way.

In some simple examples, as the real hyperbolic space $\H^n=\R^n\rtimes_{Id}\R$, the construction of non-zero elements in the $L^p$-cohomology is not difficult by using duality arguments, developped mainly in \cite{GKS2,GT10,Pa08}. (Indeed, one can easily find a differential form with a simple expression, like \eqref{FormaNoNula}, that is not zero in the $L^p$-cohomology for this particular case.) However, the Sobolev inequalities provide an obstruction to apply the same method to more complicaded cases. The main obstruction seems to be the behaviour of the forms on a neighborhood of a special point on the boundary at infinity of the groups. In order to avoid it, we propose an alternative definition of $L^p$-cohomology in the case of Gromov-hyperbolic spaces: the \textit{relative $L^p$-cohomology}. It broadly consists of taking only the differential forms that vanish on a neighborhood of some fixed point on the boundary at infinity. We denote by $L^pH^k(M,\xi)$ the $k^{th}$-space of relative $L^p$-cohomology of the pair $(M,\xi)$, where $M$ is a Gromov-hyperbolic Riemannian manifold and $\xi\in\partial M$. As in the classic version, this gives a quasi-isometry invariant under certain conditions. More precisely, we have the following result: 

\begin{theorem}\label{invarianza} Let $M$ and $N$ be two complete, uniformly contractible and Gromov-hyperbolic Riemannian manifolds with bounded geometry, and $\xi$ a fixed point in $\partial M$. If $F:M\to N$ is a quasi-isometry, then for every $p\in[1,+\infty)$ and $k\in\N$ there is an isomorphism of topological vector spaces between $L^pH^k(M,\xi)$ and $L^pH^k(N,F(\xi))$.
\end{theorem}

A metric space is \textit{uniformly contractible} if there is a function $\phi:[0,+\infty)\to [0,+\infty)$ such that every ball $B(x,r)=\{x'\in X : |x'-x|<r\}$ is contractible in the ball $B(x,\phi(r))$. We say that a Riemannian manifold has \textit{bounded geometry} if it has positive injectivity radius and its sectional curvature is uniformly bounded from above and below. 

In order to prove Theorem \ref{invarianza} we show the invariance under quasi-isometries of an equivalent discrete version of the relative $L^p$-cohomology. Once the invariance is proven, we study the vanishing and non-vanishing of the $L^p$-cohomology of a group $\R^{n-1}\rtimes_\alpha\R$ relative to a special point $\infty$ on its boundary at infinity:

\begin{theorem}\label{TeoremaLpHeintzeCalculo}
Let $G=\R^{n-1}\rtimes_{\alpha}\R$ be a purely real Heintze group, where $\alpha$ has positive eigenvalues $\lambda_1\leq\cdots\leq\lambda_{n-1}$. For $k=1,\ldots,n-1$ we write $w_k=\lambda_1+\cdots+\lambda_k$. Then 
\begin{itemize}
    \item $L^pH^k(G,\infty)=0$ for $k\geq 2$ and $p>\frac{\mathrm{tr}(\alpha)}{w_{k-1}}$;
    \item $L^p H^k(G,\infty)\neq 0$ for $k\geq 2$ and $\frac{\mathrm{tr}(\alpha)}{w_k}<p\leq \frac{\mathrm{tr}(\alpha)}{w_{k-1}}$; and
    \item $L^1H^k(G,\infty)\neq 0$ for $p>\frac{\mathrm{tr}(\alpha)}{w_1}$.
\end{itemize}
\end{theorem}

If two Heintze groups $G_1$ and $G_2$ as above are quasi-isometric, then there exists a quasi-isometry that sends the point $\infty\in\partial G_1$ to $\infty\in\partial G_2$ (see Remark \ref{lemaCornulier}). Combining this fact with Theorem \ref{invarianza} and Theorem \ref{TeoremaLpHeintzeCalculo}, we easily get Theorem \ref{ResultadoPrincipal}. Indeed, if $\R^{n-1}\rtimes_{\alpha_1}\R$ and $\R^{n-1}\rtimes_{\alpha_2}\R$ are quasi-isometric, then $\alpha_1$ and $\frac{tr(\alpha_1)}{tr(\alpha_2)}\alpha_2$ have the same eigenvalues.

Finally, in the case of degree $1$ we use the identification between $L^p$-cohomology and Besov spaces  given in \cite{BP} to prove the following result: 

\begin{theorem}\label{TeoGr1}

Let $M$ be a complete Gromov-hyperbolic Riemannian manifold with bounded geometry and $\xi\in\partial M$. Then for every $p\geq 1$ there exists a continuous embedding $L^p H^1(M,\xi)\hookrightarrow L^pH^1(M)$. Moreover, if $p>\mathrm{Cdim}_{AR}(\partial M)$, the image of this embedding is dense. 
\end{theorem}

In the previus theorem, $\mathrm{Cdim}_{AS}(\partial M)$ denotes the Ahlfors regular conformal dimension of $\partial M$, that is, the infimum of dimensions of all Ahlfors regular metrics that are cuasi-symmetric equivalent to any visual metric on $\partial M$. See for example \cite[Chapter 7]{MT} for more details. 

If $G=N\rtimes_\alpha\R$ is a Heintze group, then it is known that $L^pH^1(G)=0$ for every $p\leq \mathrm{tr}(\alpha)/w_1$, and $L^pH^1(G)\neq 0$ for $p>\mathrm{tr}(\alpha)/w_1$ (see for example \cite[Theorem 4]{Pa89a} or \cite[Corollary 1.6]{C}). Assuming that, the first part of Theorem \ref{TeoGr1} allows to conclude that $L^pH^1(G,\xi)=0$ for every $p\leq \mathrm{tr}(\alpha)/w_1$ and any $\xi\in\partial G$, which completes Theorem \ref{TeoremaLpHeintzeCalculo}. 

%\begin{theorem}\label{TeoGr1}
%Let $G=N\rtimes_\alpha\R$ be a Heintze group and $\xi\in \partial G$. Then there exists a continuous embedding from $L^pH^1(G,\xi)$ to $L^pH^1(G)$. Moreover, if $p>\mathrm{tr}(\alpha)/w_1$, the image of this embedding is dense. 
%\end{theorem}

%Combining the previous theorem with \cite[Corollary 0.2]{BP} and \cite[Corollary 1.6]{C} we have that $L^pH^1(G,\xi)=0$ for every $p\leq tr(\alpha)/w_1$ and $L^pH^1(G,\xi)\neq 0$ for every $p>tr(\alpha)/w_1$. The second part of this assertion is also independently proved in Section \ref{Sec3}. 

\subsection{Notation and conventions}

We denote by $|x-y|$ the distance between two points $x$ and $y$ belonging to any metric space. In the last section we also use the notation $d(x,y)$ to refer to a visual metric.

A quasi-isometry of constants $\lambda\geq 1$ and $\epsilon\geq 0$ from a metric space to another is a map $F:X\to Y$ such that
\begin{enumerate}
    \item[(a)] $\lambda^{-1}|x-x'|-\epsilon\leq |F(x)-F(x')|\leq \lambda|x-x'|+\epsilon\ \text{ for every }x,x'\in X$; and
    \item[(b)] for every $y\in Y$ there exists $x\in X$ such that $|F(x)-y|\leq \epsilon$.
\end{enumerate}

If $X$ is a Gromov-hyperbolic space we write $\partial X$ to mean its boundary at infinity. If $F:X\to Y$ is a quasi-isometry between two Gromov-hyperbolic spaces, the induced homeomorphism between their boundaries is denoted also by $F$. We write $\overline{F}$ to mean a quasi-inverse of $F$, that is, a quasi-isometry from $Y$ to $X$ such that $F\circ \overline{F}$ and $\overline{F}\circ F$ are at bounded uniform distance from the identity. We refer to \cite{GH} for details about quasi-isometries and Gromov-hyperbolic spaces.

By a \textit{measurable $k$-form}, or simply \textit{$k$-form}, on a smooth manifold $M$ we mean a function 
$$\omega:M\to \bigcup_{x\in M}\Lambda^k(T_x M),\ x\mapsto \omega_x\in \Lambda^k(T_x M)$$ 
whose coefficients with respect to any parametrization of $M$ are measurable functions. 
Here $\Lambda^k(T_x M)$ denotes the space of alternating $k$-linear maps on the tangent space $T_x M$. If its coefficients are in addition smooth we say that $\omega$ is a \textit{differential $k$-form}. The space of differential $k$-forms on $M$ is denoted by $\Omega^k(M)$.

Given two real functions $f$ and $g$ defined in the same domain we write $f\preceq g$ if there exists a constant $C>0$ such that $f\leq Cg$, and $f\asymp g$ if $f\preceq g$ and $g\preceq f$.

When we talk about \textit{cochain complexes}, \textit{cochain maps}, \textit{homotopies} and \textit{homotopy equivalences} between cochain complexes, we will do it in a continous sense, that is, all maps involved are continuous. 

\section{Relative $L^p$-cohomology}

Consider a Riemannian manifold $M$ of dimension $n$. We denote by $\|\ \|_x$ the Riemannian norm at the point $x\in M$ and by $dV$ the volume on $M$. 

If $\omega\in\Omega^k(M)$ and $p\geq 1$ we set
$$\|\omega\|_{L^p}=\left(\int_M |\omega|_x dV(x)\right)^{\frac{1}{p}},$$
where $|\omega|_x=\sup\bigl\{|\omega_x(v_1,\ldots,v_k)|:v_i\in T_xM\text{ with }\|v_i\|_x=1\text{ for every }i=1,\ldots,k\bigr\}.$ Then we consider $L^p\Omega(M)$ as the space of differential $k$-forms $\omega$ on $M$ with
$$|\omega|_{L^p}=\|\omega\|_{L^p}+\|d\omega\|_{L^p}<+\infty,$$ 
where $d$ is the usual exterior derivative. Observe that $d$ is continuous with this norm.

Since $\bigl(L^p\Omega(M),|\ \ |_{L^p}\bigr)$ is not complete we take its completion $L^pC^k(M)$ and extend continuously the exterior derivative. The \textit{de Rham $L^p$-cohomolgy of $M$} is the cohomology of the cochain complex $\bigl(L^pC^*(M),d\bigr)$. Observe that $L^p\Omega(M)$ is continuously included in the space of $L^p$-integrable (not necesarily smooth) $k$-forms on $M$ up to almost everywhere zero $k$-forms, denoted by $L^p(M,\Lambda^k)$, which is naturally equipped with the norm $\|\ \|_{L^p}$. Since the second space is complete, the elements of $L^pC^k(M)$ can be seen as elements of $L^p(M,\Lambda^k)$. Indeed, an equivalent definition of de Rham $L^p$-cohomology can be done in terms of $L^p$-integrable forms and weak exterior derivative (see for example \cite{GKS2,GT10}).

If $M$ is in addition Gromov-hyperbolic and complete we fix a point $\xi$ on its boundary at infinity $\partial M$ and introduce the \textit{relative de Rham $L^p$-cohomology} of the pair $(M,\xi)$ in the following way: Consider $L^pC^k(M,\xi)$ the subspace of $L^pC^k(M)$ consisting of all elements that vanish (almost everywhere) on a neighborhood of $\xi$ in $\overline{M}=M\cup\partial M$. It is clear that $d(L^pC^k(M,\xi))\subset L^pC^{k+1}(M,\xi)$ for every $k$, then we take the family of spaces
$$L^p H^k(M,\xi)=\frac{\mathrm{Ker}\ d|_{L^p C^k(M,\xi)}}{\mathrm{Im}\ d|_{L^p C^{k-1}(M,\xi)}}$$
endowed with the topology induced by $|\ \ |_{L^p}$.

We can define a version of relative $L^p$-cohomology in the context of simplicial complexes. To this end suppose that $X$ is a finite-dimensional simplicial complex with a length distance such that there exist a constant $C\geq 0$ and a function $N:[0,+\infty)\to\N$ satisfying
\begin{enumerate}
\item[(c)] all simplices in $X$ have diameter smaller than $C$; and
\item[(d)] every ball of radius $r$ intersects at most $N(r)$ simplices.
\end{enumerate}
We say that $X$ has \textit{bounded geometry} if it satisfies all these properties. The set of $k$-simplices of $X$ will be denoted by $X^{(k)}$.

For a fixed real number $p\geq 1$ we consider the \textit{simplicial} $\ell^p$\textit{-cohomology} of $X$ as the cohomology of $\bigl(\ell^p\bigl(X^{(*)}\bigr),\delta\bigr)$, where $\delta:\ell^p\bigl(X^{(k)}\bigr)\to \ell^p\bigl(X^{(k+1)}\bigr)$ is the usual coboundary operator. It is not difficult to see that it is well-defined and continuous using properties $(c)$ and $(d)$.

If $X$ is Gromov-hyperbolic and $\xi\in\partial X$, we consider $\ell^p\bigl(X^{(k)},\xi\bigr)$ the subspace of $\ell^k\bigl(X^{(k)}\bigr)$ consisting of all $k$-cochains that vanish on a neighborhood of $\xi$ in $\overline X$. (We say that $\theta\in\ell^p\bigl(X^{(k)}\bigr)$ is zero or vanishes on $U\subset\overline{X}$ if for every $k$-simplex $\sigma\subset U$ we have $\theta(\sigma)=0$.) In this case we also have  $\delta\bigl(\ell^p\bigl(X^{(k)},\xi\bigr)\bigr)\subset\ell^p\bigl(X^{(k+1)},\xi\bigr)$, then the \textit{relative simplicial $\ell^p$-cohomology of the pair $(X,\xi)$} is the family of topological vector spaces
$$\ell^p H^k(X,\xi)=\frac{\Ker \delta|_{\ell^p(X^{(k)},\xi)}}{\I \delta|_{\ell^p(X^{(k-1)},\xi)}}.$$
A de Rham-type theorem works for $L^p$-cohomology in the classical sense (\cite{G,GKS88,Pa95}), so both de Rham and simplicial versions are isomorphic for a Riemannian manifold with certain triangulation. We will see later that it is also true in the relative case. 

\begin{remark}
It is usual to consider the \textit{reduced $L^p$-cohomology} in both cases, defined by taking the quotient by the closure of the exact forms or cochains. This has the advantage of providing a family of Banach spaces that remain invariant by quasi-isometries. However, this does not make much sense in our case because the spaces of relative forms or cochains are not Banach spaces.
\end{remark}

\begin{remark}
It is possible to construct examples for which the reduce $L^p$-cohomology depends on the point on the boundary. For example one can consider $X$ the space that results from gluing $\H^n$ and $[0,+\infty)$ by identifying $0$ with any point of $\H^n$. If $\xi_0\in \partial X$ is the point represented by the geodesic ray $[0,+\infty)$ one can see that $dim\bigl(L^pH^1(X,\xi_0)\bigr)=1$, while $dim\bigl(L^pH^1(X,\xi)\bigr)=\infty$ if $\xi\neq \xi_0$ .
\end{remark}

\subsection{Quasi-isometry invariance}

We want to prove the following theorem:  

\begin{theorem}\label{invarianzaSimplicial} Let $X$ and $Y$ be two uniformly contractible and Gromov-hyperbolic simplicial complexes with bounded geometry, and $\xi$ a fixed point in $\partial X$. If $F:X\to Y$ is a quasi-isometry, then for every $p\in[1,+\infty)$ and $k\in\N$ there is an isomorphism of topological vector spaces between $\ell^pH^k(X,\xi)$ and $\ell^pH^k(Y,F(\xi))$.
\end{theorem}

The proof of Theorem \ref{invarianzaSimplicial} will be an adaptation of the proof of Theorem 1.1 in \cite{BP}.\\

Observe that every $\theta\in \ell^p\bigl(X^{(k)}\bigr)$ has a natural linear extension $\theta:C_k(X)\to\R$, where
$$C_k(X)=\left\{\sum_{i=1}^m t_i\sigma_i :  t_1,\ldots,t_m\in\R, \sigma_1,\ldots,\sigma_m\in X^{(k)}\right\}$$
is the space of $k$-chains on $X$. The \textit{support} of $c=\sum_{i=1}^m t_i\sigma_i$ in $C_k(X)$ (with $t_i\neq 0$ for all $i=1,\ldots,m$) is $|c|=\{\sigma_1,\ldots,\sigma_m\}$. We also define the \textit{uniform norm} and the \textit{length} of $c$ by
$$\|c\|_\infty=\max\{|t_1|,\ldots,|t_m|\}\text{, and } \ell(c)=m.$$

We will prove Theorem \ref{invarianzaSimplicial} in a similar way as in \cite{BP}. For this purpose we consider the following lemmas, in which we assume the setting of the theorem.

\begin{lemma}[\cite{BP}]\label{lema1} 
The quasi-isometry $F:X\to Y$ induces a family of maps $c_F:C_k(X)\to C_k(Y)$ at bounded uniform distance from $F$ which verify $\partial c_F(\sigma)=c_F(\partial\sigma)$ for every $\sigma\in X^{(k)}$. Moreover, for every $k\in\N$ there exist constants $N_k$ and $L_k$ (depending only on $k$ and the geometric data of $X,Y$ and $F$) such that 
$\|c_F(\sigma)\|_\infty\leq N_k$ and $\ell(c_F(\sigma))\leq L_k$ for every $\sigma\in X^{(k)}$ 
\end{lemma}

\begin{lemma}[\cite{BP}]\label{lema2}
If $G:X\to Y$ is another quasi-isometry at bounded uniform distance from $F$, then there exists an homotopy $h:C_k(X)\to C_{k+1}(Y)$ between $c_F$ and $c_G$. That is,
\begin{equation*}
\left\{
\begin{array}{cc}
h(v)=c_F(v)-c_G(v) & \text{ if } v\in X^{(0)},\\
\partial h(\sigma)+h(\partial\sigma)=c_F(\sigma)-c_G(\sigma) & \text{ if }\sigma\in X^{(k)},\ k\geq 1.
\end{array}
\right.
\end{equation*}
Furthermore, $\|h(\sigma)\|_\infty$ and $\ell(h(\sigma))$ are uniformly bounded by constants $N'_k$ and $L'_k$ that depend only on the geometric data of $X,Y,F$ and $G$.
\end{lemma}

\begin{proof}[Proof of Theorem \ref{invarianzaSimplicial}]

We consider the \textit{pull-back} of a k-cochain $\theta\in\ell^p\bigl(Y^{(k)},{F(\xi)}\bigr)$ as the composition $F^*\theta=\theta\circ c_F$ (it depends on the choice of $c_F$). 
Observe first that $F^*$ is well-defined and continuous from $\ell^p\bigl(Y^{(k)},F(\xi)\bigr)$ to $\ell^p\bigl(X^{(k)}\bigr)$. Indeed, Jensen's inequality allows to prove that for $\theta\in\ell^p\bigl(Y^{(k)},F(\xi)\bigr)$,
\begin{align*}
\|F^*\theta\|_{L^p}^p=\sum_{\sigma\in X^{(k)}}|\theta(c_F(\sigma))|^p \leq N_k^pL_k^{p-1}\sum_{\sigma\in X^{(k)}}\sum_{\tau\in |c_F(\sigma)|}|\theta(\tau)|^p.
\end{align*}
Since $F$ is a quasi-isometry and the distance between $c_F(v)$ and $F(v)$ is uniformly bounded for all $v\in X^{(0)}$, we can find a constant $C_k$ such that if $dist(\sigma_1,\sigma_2)>C_k$, then $c_F(\sigma_1)\cap c_F(\sigma_2)=\emptyset$. Using the bounded geometry of $X$ we have that every $\tau\in Y^{(k)}$ satisfies $\tau\in |c_F(\sigma)|$ for at most $N(C+C_k)$ simplices $\sigma\in X^{(k)}$. This implies that
\begin{align*}
\|F^*\theta\|_{L^p}^p &\leq N_k^p L_k^{p-1} N(C+C_k) \sum_{\tau\in Y^{(k)}}|\theta(\tau)|^p
= N_k^p L_k^{p-1} N(C+C_k) \|\theta\|^p_{\ell^p}.
\end{align*}

Let us prove that $F^*\theta$ is zero on some neighborhood of $\xi$ for any fixed $\theta$ in $\ell^p\bigl(Y^{(k)},F(\xi)\bigr)$. Assume that $\theta$ is zero on $V$, a neighborhood of $F(\xi)$ in $\overline{Y}$. If $\sigma\in X^{(k)}$ and $v$ is one of its vertex, then
\begin{equation}\label{distHausdorff}
dist_H\bigl(c_F(\sigma),F(v)\bigr)\leq dist_H\bigl(c_F(\sigma),c_F(v)\bigr)+dist_H\bigl(c_F(v),F(v)\bigr),    
\end{equation}
where $dist_H$ denotes the Hausdorff distance. By construction of $c_F$ the distance (\ref{distHausdorff}) is uniformly bounded by a constant $D_k$. We define $W$ as the closure in $\overline{Y}$ of the set $$\bigl\{y\in Y: dist(y,V^c\cap Y)>D_k\bigr\}.$$ 
Since $F$ is a quasi-isometry, there exists $U\subset\overline{X}$ a neighbourhood of $\xi$ such that $F(U)\subset W$. For every $k$-simplex $\sigma\subset U$ we have $c_F(\sigma)\subset V$ and hence $F^*\theta(\sigma)=0$. We conclude that $F^*\theta$ vanishes on $U$.

By definition we have $\delta F^*=F^*\delta$, which implies that $F^*$ defines a map in cohomology, denoted by $F^{\#}:\ell^pH^k(Y,{F(\xi)})\to \ell^p H^k(X,\xi)$. We have to prove that $F^{\#}$ is an isomporphism.\\

\underline{Claim}: If $F,G:X\to Y$ are two quasi-isometries at bounded uniform distance, then
$F^{\#}=G^{\#}$.\\

We need a family of continuous linear maps $H_k:\ell^p\bigl(Y^{(k)},{F(\xi)}\bigr)\to \ell^p\bigl(X^{(k-1)},\xi\bigr)$ such that
\begin{equation*}
\left\{
\begin{array}{cc}
    F^*\theta-G^*\theta=H_1\delta\theta &\text{ if }\theta\in \ell^p\bigl(Y^{(0)},{F(\xi)}\bigr),  \\
    F^*\theta-G^*\theta=H_{k+1}\delta\theta+\delta H_k\theta &\text{ if }\theta\in \ell^p\bigl(Y^{(k)},F(\xi)\bigr),\ k\geq 1 
\end{array}
\right.
\end{equation*}

It is easy to verify that both conditions are satisfied by the map
$$H_k\theta:X^{(k)}\to\R,\  H_k\theta(\sigma)=\theta(h(\sigma)),$$ 
where $h$ is the map given by Lemma \ref{lema2}. Using the same argument as for $F^*$ we can prove that $H_k\theta$ is in $\ell^p \bigl(X^{(k-1)}\bigr)$ and $H_k$ is continuous. To see that $H_k\theta$ vanishes on some neighborhood of $\xi$ (and finish the proof of the claim) observe that $h(\sigma)$ have uniformly bounded length, which implies that $dist_H\bigl(c_F(\sigma),h(\sigma)\bigr)$ is uniformly bounded.\\ 

As a consequence of the claim we have that $F^{\#}$ does not depend on the choice of $c_F$. Moreover, if $T:Y\to Z$ is another quasi-isometry, a possibe choice of the function $c_{T\circ F}$ is the composition $c_T\circ c_F$. In this case $(T\circ F)^*=F^*\circ T^*$ and hence $(T\circ F)^{\#}=F^{\#}\circ T^{\#}$.

Finally, if $\overline{F}:Y\to X$ is a quasi-inverse of $F$, then by the claim $(F\circ \overline{F})^{\#}$ and $(\overline{F}\circ F)^{\#}$ are the identity in relative cohomology. Since $(F\circ \overline{F})^{\#}= \overline{F}^{\#}\circ F^{\#}$ and $(\overline{F}\circ F)^{\#}=F^{\#}\circ\overline{F}^{\#}$, the statement follows.

\end{proof}

\subsection{Equivalence between simplicial and de Rham relative $L^p$-cohomology}

Let $M$ be a complete Gromov-hyperbolic Riemannian manifold of dimension $n$ and $\xi\in\partial M$. We assume that $M$ admits a triangulation $X_M$ with bounded geometry such that every simplex is uniformly biLipschitz diffeomorphic to the standar Euclidean simplex of the same dimension. The existence of such triangulation is guaranteed if, for example, $M$ has bounded geometry, that is, it has bounded curvature and positive injectivity radius (see \cite{A}).

For every vertex $v$ of $X_M$ we define its \textit{open star} $U(v)$ as the interior of the union of all simplices containing $v$. Observe that $\mathcal{U}=\Bigl\{U(v):v\in X^{(0)}_M\Bigr\}$ is an open covering of $M$ satisfying that every nonempty intersection $U(v_1)\cap\cdots\cap U(v_k)$ is biLipschitz diffeomorphic to the unit ball in $\R^n$ with uniform Lipschitz constant. The simplicial complex $X_M$ can be seen as the nerve of $\mathcal{U}$, that is, we can identify $X_M^{(k)}$ with the set $$\bigl\{(U_1,\ldots,U_k) : U_1\cap \cdots\cap U_k\neq\emptyset \text{ and } U_1,\ldots U_k\in\mathcal{U}\bigr\}.$$

Under the above considerations, we will prove the following result:

\begin{theorem}\label{equivalencia}
Let $M$, $X_M$ and $\xi\in \partial$ be as above, then for every $k\in\N$ and $p\in[0,+\infty)$ the spaces $L^pH^k(M,\xi)$ and $\ell^p H^k(X_M,\xi)$ are isomorphic.
\end{theorem} 

It is clear that Theorem \ref{invarianza} follows from Theorem \ref{invarianzaSimplicial} and Theorem \ref{equivalencia}. To prove Theorem \ref{equivalencia} we will use two lemmas which appear in \cite{Pa95}. We give their complete and more detailed proofs.

By a \textit{bicomplex} we mean a family of topological vector spaces $\bigl\{C^{k,\ell}\bigr\}_{k,\ell\in\N}$ together with two families of continuous linear operators $d'=d'_k:C^{k,\ell}\to C^{k+1,\ell}$ and $d''=d''_\ell:C^{k,\ell}\to C^{k,\ell+1}$ such that $d'\circ d'=0$, $d''\circ d''=0$ and $d'\circ d''+d''\circ d'=0$. 

\begin{lemma}[Lemma 5,\cite{Pa95}]\label{weyl} 
Let $\bigl(C^{*,*}, d', d''\bigr)$ be a bicomplex . Suppose that for every $\ell \in \mathbb N$
the complex $\bigl(C^{*, \ell}, d'\bigr)$ retracts to the subcomplex $\bigl(E^\ell := \mathrm{Ker}\ d'\vert _{C^{0,\ell}}\to 0 \to 0 \to \cdots\bigr)$. Then the complex $(D ^*, \delta)$, defined
by 
$$D^m = \bigoplus _{k + \ell = m} C^{k, \ell}\text{ and }\delta = d' + d'',$$ 
is homotopically equivalent to $(E^*, d'')$.
\end{lemma}

\begin{proof}
For every $K\in\N$ let $\bigl(C^{*,*}_{[K]},d',d''\bigr)$ be the subcomplex of $\bigl(C^{*,*},d',d''\bigr)$ defined by 
$$C^{k,\ell}_{[K]}=\left\{\begin{array}{ccc} C^{k,\ell}\ &\text{if }k<K\\
\mathrm{Ker}\ d'|_{C^{k,\ell}}\ &\text{if }k=K\\
0\ &\text{if }k>K\end{array}\right..$$
For every $m\in\N$ let $D^{m}_{[K]}=\bigoplus_{k+\ell=m}C^{k,\ell}_{[K]}$. It is clear that $D^{*}_{[K]} \subset D^{*}_{[K+1]}$ for every $K$ and $\cup _{K \ge 0} D^{*}_{[K]} = D^{*}$.
Moreover, by definition of $E^*$, we have $D^{*}_{[0]} = E^*$. Therefore, to prove the lemma, it will suffice to show
that $D^{*}_{[K]}$ retracts to $D^{*}_{[K-1]}$ for every $K \ge 1$.

In order to simplify the notation we set
$$\mathcal{C}_0=\bigoplus _{\ell \ge 0} C^{0, \ell},\ \ \mathcal{C}_1=\bigoplus _{k\ge 1, \ell \ge 0} C^{k, \ell},\ \text{ and }\mathcal{E}=\bigoplus _{\ell \ge 0} E ^\ell.$$
We also write $\mathcal{C}=\mathcal{C}_0\cup\mathcal{C}_1$. By assumption, for every $\ell \in \mathbb N$,
the complex $\bigl(C^{*, \ell}, d'\bigr)$ retracts to the subcomplex $\bigl(E^\ell\to 0\to 0\to\cdots\bigr)$. Thus there exist continuous operators 
$h : \mathcal{C}_1 \to \mathcal{C}$,
 $\varphi : \mathcal{C}_0 \to \mathcal{E}$, satisfying $h\bigl(C^{k,\ell}\bigr)\subset C^{k-1,\ell}$ and $\varphi\bigl(C^{0,\ell}\bigr)\subset E^\ell$, such that
\begin{enumerate}
\item[(e)] $h \circ d' = \mathrm{Id} - i \circ \varphi$ on $\mathcal{C}_0$; and
\item[(f)] $d'\circ h + h\circ d' = \mathrm{Id}$ on $\mathcal{C}_1$,
\end{enumerate}
where $i:\mathcal{E}\to \mathcal{C}_0$ is the inclusion.
We extend $h$ to the whole space $\mathcal{C}$ by letting $h = 0$ on
$\mathcal{C}_0$.

Define $b : \mathcal{C} \to \mathcal{C}$  by 
$$b=\left\{\begin{array}{cc}  -(d'' \circ h + h \circ d'') & \text{ on }\mathcal{C}_1\\  i \circ \varphi &\text{ on }\mathcal{C}_0 
\end{array}\right..$$
By $(e)$ and $(f)$ we have the equality   
$\delta \circ h + h \circ \delta = \mathrm{Id} - b$ on $\mathcal{C}$.
This implies in particular that $b$ commutes with $\delta$.

We are now ready to show that $D^{*}_{[K]}$ retracts to $D^{*}_{[K-1]}$ for every $K \ge 1$.  
Observe that $h$ sends $D^m_{[K]}$ to  $D^{m-1}_{[K]}$, then we can consider $h_{[K]} : D ^{m}_{[K]} \to D ^{m-1}_{[K]}$ the induced operator. The map $b$ satisfies $b(C^{k, \ell}) \subset C^{k-1,\ell +1}$ for $k \ge 1$ and $b(C^{0, \ell}) \subset C^{0, \ell}$.
Moreover, for $K \ge 1$, one has 
$$b\bigl(\mathrm{Ker} d'\vert _{C^{K,\ell}}\bigr) \subset \mathrm{Ker} d' \vert _{ C^{K-1,\ell+1}}.$$
Indeed, if $d'\omega = 0$, then one has also $d'd''\omega = 0$. The definition of $b$ and the relation (f) gives:
$$d'b\omega = -(d'd''h\omega + d'hd''\omega) = d''d'h\omega - d'hd''\omega = -d''\omega + d''\omega = 0.$$
Therefore $b$ sends every $D^m_{[K]}$ to $D^m_{[K-1]}$ for $K \ge1$. Let $b_{[K]} : D^*_{[K]} \to D^* _{[K-1]}$ be the induced operator.
As we saw above, it commutes with $\delta$. Since $\delta \circ h + h \circ \delta = \mathrm{Id} - b$
on $\mathcal{C}$, we get
$$\delta \circ h_{[K]} + h_{[K]} \circ \delta = \mathrm{Id} - i _{[K-1]} \circ b_{[K]}$$
and also 
$$\delta \circ h_{[K-1]} + h_{[K-1]} \circ \delta = \mathrm{Id} - b_{[K]} \circ i _{[K-1]},$$
where $i _{[K-1]} : D^* _{[K-1]} \to D^*_{[K]}$ is the inclusion. All these maps are continuous, then the lemma follows.
\end{proof}

\begin{lemma}[Lemma 8 in \cite{Pa95}]\label{aciclico}
Let $B$ be the unit ball in $\R^n$, then the cochain complex $\bigl(L^pC^*(B),d\bigr)$ retracts to the complex $\bigl(\R\to 0\to 0\to\cdots\bigr).$
\end{lemma}

In the proof of Lemma \ref{aciclico} we will use the following version of the Leibniz Integral Rule, whose proof follows directly from the classic version or applying the Dominated Convergence Theorem. 

\begin{lemma}\label{Leibniz}
Consider a measure space $(Z,\mu)$ and a family of differential $k$-forms $\{\omega(z)\}_{z\in Z}$
defined on an open set $U\subset\R^n$.
Supose that every coefficient of $\omega(z)_x$ is integrable on $z$ and each of its iterated partial derivatives of any order is dominated by a function in $L^1(Z)$. Then the $k$-form $\omega$ defined by
$$\omega_x(u_1,\ldots,u_k)=\left(\int_Z \omega(z)_x d\mu(z)\right)(u_1,\ldots,u_k)=\int_Z \omega(z)_x(u_1,\ldots,u_k)d\mu(z)$$
belongs to $\Omega^k(M)$ and its derivative is $d\omega=\int_Z d\omega(z) d\mu(z).$
\end{lemma}

\begin{proof}[Proof of Lemma \ref{aciclico}]
For a fixed $x\in B$ we consider $\varphi_x:[0,1]\times B\to B$, $\varphi_x(t,y)=ty+(1-t)x$ and $\eta_t:B\to [0,1]\times B$, $\eta_t(y)=(t,y)$. Denote $\frac{\partial}{\partial t}=(1,0)\in [0,1]\times B$ and define the $\frac{\partial}{\partial t}$-contraction of a $k$-form $\omega$ as the $(k-1)$-form 
$$\iota_{\frac{\partial}{\partial t}}\omega_y (u_1,\ldots,u_{k-1})=\omega_y\left(\mathsmaller{\frac{\partial}{\partial t}},u_1,\ldots,u_{k-1}\right).$$

Now we define, for $y\in B$, $\omega\in L^\phi\Omega^k(B)$ and $u_1,\ldots,u_{k-1}$,
\begin{align*}
\chi_x(\omega)=\int_0^1 \eta_s^*\left(\iota_{\frac{\partial}{\partial t}}\varphi_x^*\omega\right)ds.
\end{align*}
Observe that the coefficients of 
$\eta_s^*\Bigl(\iota_{\frac{\partial}{\partial s}}\varphi_x^*\omega\Bigr)_y$ are smooth on $t,x$ and $y$, thus we can apply Lemma \ref{Leibniz} to prove that  $\chi_x(\omega)$ is a differential $(k-1)$-form.

For every $(k-1)$-simplex $\sigma$ we have
\begin{align*}
\int_\sigma \chi_x(\omega) 
= \int_\sigma\int_0^1 \eta_s^*\left(\iota_{\frac{\partial}{\partial t}}\varphi_x^*\omega\right)ds
= \int_{[0,1]\times\sigma}\varphi_x^*\omega
= \int_{\varphi([0,1]\times\sigma)}\omega
= \int_{C_\sigma}\omega,
\end{align*}
where he cone $C_\sigma$ is defined as follows: If $\sigma=(x_0,\ldots,x_{k-1})$, then $C_\sigma=(x,x_0,\ldots,x_{k-1})$.

Suppose that $\sigma$ is a $k$-simplex in $B$ with $\partial \sigma=\tau_0+\cdots+\tau_k$ and $\omega\in \Omega^k(B)$. Using Stoke's theorem we have
\begin{align*}
\int_\sigma \chi_x(d\omega) = \int_{C_\sigma} d\omega = \int_{\partial C_\sigma}\omega = \int_\sigma \omega-\sum_{i=0}^k \int_{C_{\tau_i}}\omega =\int_\sigma \omega- \int_{\partial\sigma}\chi_x(\omega)=\int_\sigma \omega- \int_{\sigma}d\chi_x(\omega).   
\end{align*}
Since this is true for every $k$-simplex, we conclude that 
\begin{equation}\label{dos}
\chi_x d+d\chi_x=Id    
\end{equation}
(see for example \cite[Chapter IV]{W}). Observe that if $\omega$ is closed, then $\chi_x(\omega)$ is a primitive of $\omega$, thus this proves the classic Poincaré's lemma. However, in our case we need a primitive in $L^p$, so we will take a convenient average. 

\medskip

\medskip

Define
$$h(\omega)=\frac{1}{\Vol\left(\frac{1}{2}B\right)}\int_{\frac{1}{2}B}\chi_x(\omega)dx,$$
where $\frac{1}{2}B=B\left(0,\frac{1}{2}\right)$.

Since the coefficients of $(x,y)\mapsto \chi_x(\omega)_y$ are smooth in both variables we can use again Lemma \ref{Leibniz} to see that $h(\omega)$ is in $\Omega^k(B)$. Notice that this works because we take the integral on a ball with closure included in $B$. Moreover, the derivative of $h$ is
\begin{equation*}
dh(\omega)=\frac{1}{\Vol\left(\frac{1}{2}B\right)}\int_{\frac{1}{2}B}d\chi_x(\omega)dx.
\end{equation*}
Using (\ref{dos}) we have
\begin{equation}\label{homot}
dh(\omega)+h(d\omega)=\omega
\end{equation}
for every $\omega\in L^p\Omega^k(B)$ with $k\geq 1$.

We want to prove that $h$ is well-defined and continuous from $L^p\Omega^k(B)$ to $L^p\Omega^{k-1}(B)$.
To this end we first bound $|\chi_x(\omega)|_y$ for $y\in B$ and $\omega\in\Omega^k(B)$. Since $\iota_{\frac{\partial}{\partial t}}\varphi^*\omega$ is a form on $[0,1]\times B$ that is zero in the direction of $\frac{\partial}{\partial t}$, we have 
$|\eta_s^*(\iota_{\frac{\partial}{\partial t}}\varphi^*\omega)|_y=|\iota_{\frac{\partial}{\partial t}}\varphi^*\omega|_{(t,y)}$ for every $t\in (0,1)$ and $y\in B$. A direct calculation gives us the estimate $$|\iota_{\frac{\partial}{\partial t}}\varphi^*\omega|_{(t,y)}\leq t^{k-1}|y-x||\omega|_{\varphi(t,y)}.$$ 
Therefore, using the assumption that $t\in (0,1)$, we can write
\begin{equation}\label{uno}
|\chi(\omega)|_y\leq \int_0^1|y-x||\omega|_{\varphi(t,y)}dt.
\end{equation}

Consider the function $u:\R^n\to\R$ defined by $u(z)=|\omega|_z$ if $z\in B$ and $u(z)=0$ in the other case. Using (\ref{uno}) and the change of variables $z=ty+(1-t)x$, we have
\begin{align*}
\Vol\left(\mathsmaller{\frac{1}{2}}B\right)|h(\omega)|_y		
&\leq \int_{B\left(ty,\frac{1-t}{2}\right)}\int_0^1|z-y|u(z)(1-t)^{-n-1}dtdz \\
%&\leq \int_{B(y,2)}\int_0^1 \mathds{1}_{B(ty,1-t)}(z)|z-y|u(z)(1-t)^{-n-1}dtdz\\
&=  \int_{B(y,2)}|z-y|u(z)\left(\int_0^1 \mathds{1}_{B\left(ty,\frac{1-t}{2}\right)}(z)(1-t)^{-n-1} dt\right)dz. 	
\end{align*}
Observe that $\mathds{1}_{B\left(ty,\frac{1-t}{2}\right)}(z)=1$ implies that $|z-y|\leq 2(1-t)$. Thus,
$$\int_0^1\mathds{1}_{B\left(ty,\frac{1-t}{2}\right)}(z)(1-t)^{-n-1}dt\leq \int_0^{1-\frac{1}{2}|z-y|}(1-t)^{-n-1}dt= \int_{\frac{1}{2}|z-y|}^1 r^{-n-1}dr\preceq \frac{1}{|z-y|^n}.$$
This implies
$$Vol(\mathsmaller{\frac{1}{2}}B)|h(\omega)|_y\preceq \int_{B(y,2)}|z-y|^{1-n}u(z)\,dz.$$

Using that $\int_{B(y,2)}|z-y|^{1-n}dz$ is finite and Jensen's inequality we obtain
$$|h(\omega)|_y^p \preceq \int_{B(y,2)}|z-y|^{1-n}u(z)^p dz.$$
Therefore,
\begin{align*}
\|h(\omega)\|^p_{L^p}	&\preceq \int_B \int_{B(y,2)}|z-y|^{1-n}u(z)^p dzdy
&\preceq \int_{B(0,3)}u(z)^p\left(\int_B \frac{dy}{|z-y|^{n-1}}\right)dz\preceq \|\omega\|_{L^p}^p.
\end{align*}
By the identity $dh(\omega)=\omega-h(d\omega)$ we also have 
$$\|dh(\omega)\|_{L^p}\leq \|\omega\|_{L^p}+\|h(d\omega)\|_{L^p}\preceq \|\omega\|_{L^p}+\|d\omega\|_{L^p}=|\omega|_{L^p}.$$
We conclude that $h$ is well-defined and bounded for $k\geq 1$.

If $\omega=df$ for certain function $f$ we observe that
$$\eta_s^*\left(\iota_{\frac{\partial}{\partial t}}\varphi^*_x df\right) (y)= df_{\varphi_x(s,y)}(y-x)=(f\circ\alpha)'(s),$$
where $\alpha$ is the curve $\alpha(s)=\varphi_x(s,y)$. Then $\chi_x(df)(y)=f(y)-f(x)$, from which we get
$$h(df)=f-\frac{1}{Vol\left(\mathsmaller{\frac{1}{2}}B\right)}\int_{\frac{1}{2}B} f.$$
We define $h:L^p\Omega^0(B)\to L^p\Omega^{-1}(B)=\R$ 
by
$$h(f)=\frac{1}{Vol\left(\mathsmaller{\frac{1}{2}}B\right)}\int_{\frac{1}{2}B} f,$$
which is clearly continuous because $\frac{1}{2}B$ has finite Lebesgue measure. Then we have the identity
\begin{equation}\label{homot2}
h(f)+h(df)=f.
\end{equation}

Note that, since $h$ is bounded, then it can be extended continuously to $L^pC^k(B)$ for every $k\geq 0$. The identities \eqref{homot} and \eqref{homot2} are also true for every $\omega\in L^pC^k(B)$, which finishes the proof.
\end{proof}

\begin{proof}[Proof of Theorem \ref{equiv}]
Our strategy is to construct a convenient bicomplex and apply Lemma \ref{weyl}. To that end consider $\mathcal{U}$ the covering consisting of the open stars of the vertices of $X_M$. For a $\ell$-simplex $\Delta=(v_0,\ldots,v_\ell)$ we write
$U_\Delta=U(v_0)\cap\cdots\cap U(v_\ell)$.
Then for $k,\ell\geq 1$ we take the space $C_\xi^{k,\ell}$ consisting of functions of the form
$$\omega=\left\{\omega_\Delta\right\}_{\Delta\in X^{(\ell)}_M}\in \prod_{\Delta\in X^{(\ell)}_M} L^p C^k(U_\Delta)$$
such that
\begin{enumerate}
\item[(g)] If $i,j=0,\ldots,k$, $i\neq j$, then $\omega_{(v_0,\ldots, v_i,\ldots, v_j,\ldots, v_\ell)}=-\omega_{(v_0,\ldots,v_j,\ldots,v_i,\ldots, v_\ell)}$,
\item[(h)] $\sum_{\Delta\in X_M^{(\ell)}}\left(\|\omega_\Delta\|^p_{L^p}+\|d\omega_\Delta\|^p_{L^p}\right)<+\infty$, and
\item[(i)] There exists a neighbourhood $V$ of $\xi$ in $\overline{M}$ such that if $U_\Delta\subset V$, then $\omega_\Delta=0$ almost everywhere. 
\end{enumerate}
We equip $C_\xi^{k,\ell}$ with the norm 
$$\|\omega\|_p=\left(\sum_{\Delta\in  X_M^{(\ell)}}\|\omega_\Delta\|_{L^p}^p\right)^{\frac{1}{p}}+\left(\sum_{\Delta\in  X_M^{(\ell)}}\|d\omega_\Delta\|_{L^p}^p\right)^{\frac{1}{p}}.$$

Then we define the derivatives $d':C_\xi^{k,\ell}\to C_\xi^{k+1,\ell}$ and $d'':C_\xi^{k,\ell}\to C_\xi^{k,\ell+1}$:
\begin{itemize}
\item If $\omega\in C_\xi^{k,\ell}$, then $(d'\omega)_\Delta=(-1)^\ell d\omega_\Delta$.
\item If $\omega\in C_\xi^{k,\ell}$ and  $\Delta=(v_0,\ldots, v_{\ell+1})$, then
$$(d''\omega)_\Delta=\sum_{i=0}^{\ell+1}(-1)^i\omega_{\partial_i\Delta}|_{U_\Delta}.$$
Where $\partial_i\Delta=(v_0,\ldots,\hat{v_i},\ldots,v_{\ell+1})$. 
\end{itemize}
It is easy to show that $d'$ and $d''$ are well-defined and continuous and satisfy $d'\circ d'=0$, $d'\circ d'=0$ and $d'\circ d''+d''\circ d''=0$, then $\bigl(C_\xi^{*,*},d',d''\bigr)$ is a bicomplex.

Observe that the elements of $\mathrm{Ker}\ d'|_{C_\xi^{0,\ell}}$ are the functions $g\in C_\xi^{0,\ell}$ satisfying that $g_\Delta$ is essentially constant in $U_\Delta$ for every $\Delta\in X_M^{(\ell)}$.
Using the fact that every $U_\Delta$ is biLipschitz diffeomorphic (with uniform Lipschitz constant) to the unit ball in $\R^n$ we conclude that $\mathrm{Ker}\ d'|_{C_\xi^{0,\ell}}$ is isomporhic to $\ell^p C^\ell(X_M,\xi)$ and $d''$ coincides with the derivative on this space.

On the other hand, the elements of $\mathrm{Ker}\ d''|_{C_\xi^{k,0}}$ are of the form $\omega=\{\omega_v\}_{v\in X_M^{(0)}}$ with 
$$\omega_{v_1}|_{U(v_1)\cap U(v_2)}=\omega_{v_2}|_{U(v_1)\cap U(v_2)}$$
almost everywhere. We can take a $k$-form $\tilde{\omega}$ in $L^pC^k(M)$ such that $\tilde{\omega}|_{U(v)}=\omega_v$ almost everywhere for every $v\in X_M^{(0)}$. This $k$-form is zero on some neighborhood of $\xi$, then there is an isomorphism between $\mathrm{Ker}\ d''|_{C_\xi^{k,0}}$ and $L^p C^k(M,\xi)$ for which $d'$ coincides with the derivative on the second space. It is clear that $\|\omega\|_p\asymp |\tilde{\omega}|_{L^p}$ because every point of $M$ is in at most $n$ elements of $\mathcal{U}$.\\ 

\underline{Claim 1}: For a fixed $\ell$ the complex $\bigl(C_\xi^{*,\ell},d'\bigr)$ retracts to $\bigl(\mathrm{Ker}\ d'|_{C_\xi^{0,\ell}}\to 0\to 0\to\cdots\bigr)$.\\

We take the family of maps $h:L^p C^k(B)\to L^p C^{k-1}(B)$ given by Lemma \ref{aciclico} (where $L^p C^{-1}(B)=\R$), and for every  $\Delta\in X_M^{(\ell)}$ we consider a $L$-biLipschitz dffeomorphism $f_\Delta:U_\Delta\to B$ (which $L$ does not depend on $\Delta$). Then we define $H:C_\xi^{k,\ell}\to C_\xi^{k-1,\ell}$ by
$$(H\omega)_\Delta=(-1)^\ell f_U^* h\bigl(f_\Delta^{-1}\bigr)^*\omega_\Delta,$$
where $f_\Delta^{*}:L^pC^*(B)\to L^pC^*(U_\Delta)$ is the continuous extension of the usual pull-back of differential forms.
We write $C_\xi^{-1,\ell}:= \mathrm{Ker}\ d'|_{C_\xi^{0,\ell}}$.

It is easy to prove that $f_\Delta^*$ and $\bigl(f^{-1}_\Delta\bigr)^*$ are continuous for every $U$ by using the Lipschitz condition; hence, using also the definition of $h$, we have that $H$ is the continuous retraction we wanted.\\

\underline{Claim 2}: For a fixed $k$ the complex $\bigl(C_\xi^{k,*},d''\bigr)$ retracts to $\bigl(\mathrm{Ker}\ d''|_{C_\xi^{k,0}}\to 0\to 0\to \cdots\bigr)$.\\

We have to construct a family of bounded linear maps $\kappa:C_\xi^{k,\ell}\to C_\xi^{k,\ell-1}$ ($\ell\geq 0$), where $C_\xi^{k,-1}=\mathrm{Ker}\ d''|_{C_\xi^{k,0}}$,  such that
$$\left\{\begin{array}{cc}
    \kappa d''\omega+d''\kappa\omega=\omega & \text{ for every } \omega\in C_\xi^{k,\ell},\ \ell\geq 1,\\
    \kappa d'' g +\kappa g = g &\text{ for every }g\in C_\xi^{k,0}
\end{array}\right.$$

Let $\{\eta_v\}_{v\in X_M^{(0)}}$ be a partition of unity with respect to $\mathcal{U}$. Observe that we can take it so that $|d\eta_v|_x$ is uniformly bounded independently from $v$ and $x$. If $\ell\geq 1$ and $\omega\in C_\xi^{k,\ell}$, then for $\Delta\in X_M^{(\ell-1)}$ we define
$$(\kappa\omega)_\Delta=(-1)^{\ell}\sum_{v\in X_M^{(0)}}\eta_v\omega_{\Delta v}.$$
Where if $\Delta=(v_0,\ldots,v_{\ell+1})$, then $\Delta v=(v_0,\ldots,v_{\ell+1},v)$. Observe that if $v$ ia a vertex of $\Delta$, then $\omega_{\Delta v}=0$ because of condition (g).
For $\omega\in C_\xi^{k,0}$ and $v\in X_M^{(0)}$ we put
$$(\kappa\omega)_{v_0}=\sum_{v\in X_M^{(0)}}\eta_v\omega_v|_{U_{v_0}}.$$

A direct calculation shows that $\kappa$ is as we wanted.\\

Finally, applying Lemma \ref{weyl} we obtain that $(D^*,\delta)$ is homotopically equivalent to $\bigl(\mathrm{Ker}\ d'|_{C_\xi^{0,*}},d''\bigr)$ and $\bigl(\mathrm{Ker}\ d''|_{C_\xi^{*,0}},d'\bigr)$. The proof ends using the above identifications.

\end{proof}

\begin{remark}\label{suavizar}
Observe that the previous argument can be done considering smooth forms, so one can show that the cochain complexes $\bigl(L^p\Omega^*(M),d\bigr)$, $\bigl(L^p C^*(M),d\bigr)$ and $\bigl(\ell^p (X_M^{(*)}),\delta\bigr)$ are  all homotopically equivalent. This implies that the relative $L^p$-cohomology can be described by using smooth forms.
\end{remark}

\section{Construction of non-zero cohomology classes on Heintze groups}\label{Sec3}

\subsection{A duality idea}\label{SeccionDualidad}

In \cite{GKS2} and \cite{GT10} the following fact is proved: If $M$ is a complete and orientable $n$-dimensional Riemannian manifold, then for every $p\in(1,+\infty)$ and $k=0,\ldots,n$, the dual space of $L^p\overline{H}^k(M)$ is isometric to $L^q\overline{H}^{n-k}(M)$, where $\frac{1}{p}+\frac{1}{q}=1$. The isometry is induced by the pairing $\langle\text{ , }\rangle:L^p(M,\Lambda^k)\times L^q(M,\Lambda^{n-k})\to\R$,
\begin{equation}\label{pairing}
\langle \omega, \beta\rangle=\int_M \omega\wedge\beta,
\end{equation}
which is well-defined by Hölder's inequality. The proof uses that $L^p(M,\Lambda^k)$ and $L^q(M,\Lambda^{n-k})$ are Banach spaces. Other duality-type arguments using weaker hypothesis can be read in \cite{GT98,GT06,Pa08}.

In the relative case we have to find a natural pairing for $L^p\Omega^k(M,\xi)$ (or $L^p C^k(M,\xi)$). The answer seems to be related to the idea of local cohomology, which can be found in \cite{C}. Let us see the following definition:  Consider $M$ a complete and orientable Gromov-hyperbolic Riemannian manifold and $\xi$ a point in $\partial M$. A differential $m$-form $\beta$ on $M$ is \textit{locally} $L^q$-\textit{integrable with respect to} $\xi$ if for every $V\subset\overline{M}$ closed neighborhood of $\xi$, we have
$$\|\beta\|_{L^q,M\setminus V}=\left(\int_{M\setminus V}|\beta|^q_xdx\right)^{\frac{1}{q}}<+\infty.$$
We denote by $L^q_{\mathrm{loc}}\Omega^m(M,\xi)$ the space of all differential $m$-forms which are locally $L^q$-integrable with respect to $\xi\in\partial M$. Observe that Hölder's inequality implies that the bi-linear pairing \begin{equation}\label{pairingg}
\langle\text{ , }\rangle:L^p\Omega^k(M,\xi)\times L^q_{\mathrm{loc}}\Omega^{n-k}(M,\xi)\to\R    
\end{equation} 
is well-defined by the expression \eqref{pairing} if $\frac{1}{p}+\frac{1}{q}=1$. This allows to consider the induced linear transformations $\mu_\omega:L^q_{\mathrm{loc}}\Omega^{n-k}(M,\xi)\to\R$, $\mu_\omega=\langle\omega,\cdot\rangle$ and $\nu_\beta:L^p\Omega^k(M,\xi)\to\R$, $\nu_\beta=\langle\cdot,\beta\rangle$. We will use these maps to construct non-zero classes in the relative $L^p$-cohomology of Heintze groups.

\subsection{Proof of Theorem \ref{TeoremaLpHeintzeCalculo}}

Let $G=\R^{n-1}\rtimes_{\alpha}\R$ be a purely real Heintze group, where $\alpha$ has positive eigenvalues $\lambda_1\leq\cdots\leq\lambda_{n-1}$. Remember the notation $w_k=w_k(\alpha)=\lambda_1+\cdots+\lambda_k$ for $k=1,\ldots,n-1$; we also write $w_0=0$. The product on $G$ is given by 
$$(x,t)\cdot (y,s)=(x+e^{t\alpha} y,t+s).$$
We denote by $L_{(x,t)}$ the left translation by $(x,t)$ on $G$.

If $\langle\ ,\ \rangle_0$ is an inner product on $T_0G$ such that the factors $\R^{n-1}$ and $\R$ are orthogonal, then it determines a unique left-invariant Riemannian metric on $G$ given by
\begin{align*}
\langle (v_1,v_2),(w_1,w_2) \rangle_{(x,t)}
&=\bigl\langle (d_0L_{(x,t)})^{-1}(v_1,v_2),(d_0L_{(x,t)})^{-1}(w_1,w_2) \bigr\rangle_0\\
&=\langle e^{-t\alpha}v_1,e^{-t\alpha}w_1 \rangle_0+\lambda v_2w_2, 
\end{align*}
for every $v_1,w_1\in\R^{n-1}$, $v_2,w_2\in\R$ and $\lambda$ a fixed positive real number (any other left-invariant Riemannian metric is biLipsichitz equivalent to this one). In particular, if $v$ is a \textit{horizontal vector} in $T_{(x,t)}G$ (i.e. $v=(v_1,0)$), then the norm associated to $\langle \ ,\  \rangle_{(x,t)}$ of $v$ is 
$$\|v\|_{(x,t)}=\|e^{-t\alpha}v\|_0.$$

For every $x\in\R^{n-1}$ the curve $t\mapsto (x,t)$ is a geodesic on $G$. All these vertical geodesics are asymptotic to the future and so they define a unique point on the boundary, denoted by $\infty$ . Moreover, any other point of the boundary can be represented by one of these geodesics to the past. As a consequence we can identify the boundary $\partial G$ with $\R^{n-1}\cup\{\infty\}$.

\begin{remark}\label{lemaCornulier}
The group $\R^{n-1}$ acts by isometries on $G$, thus the group $QI(G)$ of self quasi-isometries of $G$ acts transitively on $\partial G\setminus\{\infty\}$. This implies that the action of $QI(G)$ satisfy either
\begin{itemize}
    \item $QI(G)$ acts transitively on $\partial G$, or
    \item $\infty$ is fixed by $QI(G)$.
\end{itemize}
If two Heintze groups as above are quasi-isometric, then they must satisfy simultaneously either the first or the second condition. In both cases we can observe that there exists a quasi-isometry between them that preserves the point $\infty$. A more general version of this result is proved in \cite[Lemma 6.D.1]{Cor}.
\end{remark}

Observe that a neighborhood system for the point $\infty\in\partial G$ is given by the compactification in $\overline{G}$ of the sets of the form $G\setminus \bigl(B_R\times[T,-\infty)\bigr)$, where $B_R=B(0,R)\in \R^{n-1}$ for $R>0$, and $T\in\R$.

We rename the eigenvalues of $\alpha$ by $\mu_1<\cdots<\mu_d$ and fix a Jordan basis of $\R^{n-1}$, 
$$\mathcal{B}=\{b_{ij}^\ell:i=1,\ldots,d; j=1,\ldots,r_i ;\ell=1,\ldots,m_{ij}\},$$ 
where $r_i$ is the dimension of the $\mu_i$-eigenspace, spanned by $\bigl\{b_{i1}^1,\ldots,b_{ir_i}^1\bigr\}$, $m_{ij}$ is the size of the $j$-Jordan subblock associated to $\mu_i$, and $\alpha(b_{ij}^\ell)=\mu_i b_{ij}^\ell+b_{ij}^{\ell-1}$ for every $\ell=2,\ldots,m_{ij}$. We can write
\begin{equation}\label{descomposicion}
\R^{n-1}=\bigoplus_{i,j}V_{ij},\text{ where }V_{ij}=\mathrm{Span}(\{b^\ell_{ij} : \ell=1,\ldots,m_{ij}\}).
\end{equation}

Let us denote by $\frac{\partial}{\partial t}$ the unit positive vector that spans the factor $\R$ of $G$ and by $dt$ the $1$-form associated to $\frac{\partial}{\partial t}$.
The $1$-forms associated to the dual basis of $\mathcal{B}$ are denoted by $dx_{ij}^\ell$. We put on $G$ the left-invariant Riemannian metric making the basis $\mathcal{B}\cup\bigl\{\frac{\partial}{\partial t}\bigr\}$ orthonormal in $T_0 G$.

Observe that 
$$e^{t\alpha}b_{ij}^\ell=e^{t\mu_i}\left(b_{ij}^\ell+tb_{ij}^{\ell-1}+\ldots+\frac{t^{\ell-1}}{(\ell-1)!}b_{ij}^1\right),$$
which implies
$$L^*_{(x,t)}dx_{ij}^\ell=e^{t\mu_i}\left(dx_{ij}^{\ell}+\ldots+\frac{t^{m_{ij}-\ell}}{(m_{ij}-\ell)!}dx_{ij}^{m_{ij}}\right).$$

For every $k=1,\ldots,n-1$ we denote by $\mathcal{I}_k$ the set of multi-indices 
\begin{equation}\label{FormasIndices}
I=(i_1,\ldots,i_k,j_1,\ldots,j_k,\ell_1,\ldots,\ell_k) 
\end{equation}
with $i_h=1,\ldots,d$, $j_h=1,\ldots,r_{i_h}$ and $\ell_h=1,\ldots,m_{i_hj_h}$ for every $h=1,\ldots,k$. We also assume that the function $h\mapsto (i_h,j_h,\ell_h)$ is injective and preserves the lexicographic order. For a multi-index as (\ref{FormasIndices}) we write
$$dx_I=dx_{i_1j_1}^{\ell_1}\wedge\ldots\wedge dx_{i_kj_k}^{\ell_k},\text{ and }w_I=\mu_{i_1}+\cdots+\mu_{i_k}.$$
Consider in $\mathcal{I}_1$ the lexicographic order and $\zeta:\mathcal{I}_1\to \{1,\ldots,n-1\}$ the order-preserving bijection. We denote $dx_h=dx_{i j}^\ell$ if $h=\zeta(i,j,\ell)$. We also write $dx_n=dt$.

\begin{lemma}\label{lemaformasgeneral}
\begin{enumerate}
\item[(i)] For every $I\in\mathcal{I}_k$ there exists a positive polynomial $P_I$ such that
$$|dx_I|_{(x,t)}\asymp e^{tw_I}\sqrt{P_I(t)}.$$

\item[(ii)] The volume form on $G$ is 
$dV_{(x,t)}=e^{-t \mathrm{tr}(\alpha)}dx_1\wedge\cdots\wedge dx_{n}.$
\end{enumerate}
\end{lemma} 

A polynomial $P$ is \textit{positive} if $P(t)>0$ for all $t\in \R$.

\begin{proof} 

$(i)$ On $\Lambda^k(T_0G)$ we consider the inner product $\langle\langle\text{ , }\rangle\rangle_0$ making the basis $\{dx_I:I\in\Delta_k\}$ orthonormal, thus for  $\beta,\gamma\in\Lambda^k(T_{(x,t)}G)$ we put 
\begin{equation}\label{ref1}
\langle\langle\beta,\gamma\rangle\rangle_{(x,t)}=\langle\langle L_{(x,t)}^*\beta,L_{(x,t)}^*\gamma\rangle\rangle_0.    
\end{equation}
This means that the inner product is left-invariant. 

The left-invariant norm induced by \eqref{ref1} is denoted by $[\ \ ]_{(x,t)}$. Since the operator norm $|\ \ |_{(x,t)}$ is also left-invariant, there exists a constant $C\geq 1$, independent from the point $(x,t)\in G$, such that
$C^{-1}|\ \ |_{(x,t)}\leq [\ \ ]_{(x,t)}\leq C|\ \ |_{(x,t)}.$
As a consequence it is enough to prove $(i)$ for $[\ \ ]_{(x,t)}$:
\begin{align*}
[dx_{i_1j_1}^{\ell_1}\wedge\ldots\wedge dx_{i_kj_k}^{\ell_k}]^2_{(x,t)}  &= [(L^*_{(x,t)}dx_{i_1j_1}^{\ell_1})\wedge\ldots\wedge (L^*_{(x,t)}dx_{i_kj_k}^{\ell_k})]^2_0\\
& =e^{2t(\mu_{i_1}+\ldots+\mu_{i_k})}\left[\left(dx_{i_1j_1}^{\ell_1}+\ldots+\frac{t^{m_{i_1j_1}-\ell_1}}{(m_{i_1j_1}-\ell_1)!}dx_{i_1j_1}^{m_{i_1j_1}}\right)\wedge\right.\\
&\left. \left. \ldots\wedge \left(dx_{i_kj_k}^{\ell_k}+\ldots+\frac{t^{m_{i_kj_k}-\ell_k}}{(m_{i_kj_k}-\ell_k)!}dx_{i_kj_k}^{m_{i_kj_k}}\right)\right]^2_0\right. 
\end{align*}
From this expression it is easy to extract the polynomial $P_I$.
 
$(ii)$ Here it is enough to prove that $dV_{(x,t)}(v_1,\ldots,v_n)=1$ for some positive orthonormal basis $\{v_1,\ldots,v_n\}\subset T_{(x,t)}G$. Since $\mathcal{B}\cup \{\frac{\partial}{\partial t}\}$ is orthonormal in $T_0G$, the basis
$$\mathcal{B}_{t}\cup \left\{\mathsmaller{\frac{\partial}{\partial t}}\right\}=\{d_0L_{(x,t)}(b_{ij}^\ell):i=1,\ldots,d; j=1,\ldots,r_i ;l=1,\ldots,m_{ij}\}\cup \left\{\mathsmaller{\frac{\partial}{\partial t}}\right\}$$
$$=\left\{e^{t\mu_i}\left(b_{ij}^\ell+\ldots+\frac{t^{\ell-1}}{(\ell-1)!}b^1_{ij}\right):i=1,\ldots,d; j=1,\ldots,r_i ;\ell=1,\ldots,m_{ij}\right\}\cup \left\{\mathsmaller{\frac{\partial}{\partial t}}\right\}$$
is orthonormal in $T_{(x,t)}G$. Then we can check the equality evaluating $dV_{(x,t)}$ in the elements of $\mathcal{B}_t\cup\left\{\frac{\partial}{\partial t}\right\}$.

\end{proof}

Let $V$ be the vertical vector field defined by $V(x,t)=\frac{\partial}{\partial t}$, and $\varphi_t(x,s)=(x,s+t)$ its associated flow. We say that a $k$-form $\omega$ is \textit{horizontal} if $\iota_V\omega=0$. Observe that if
\begin{equation}\label{horizontal}
\omega=\sum_{1\leq i_1<\ldots<i_k\leq n}a_{i_1,\ldots,i_k}dx_{i_1}\wedge\ldots\wedge dx_{i_k},    
\end{equation}
then $\omega$ is horizontal if and only if all coefficients $a_{i_1,\ldots,i_{k-1},n}$ are zero.

To prove that the relative $L^p$-cohomology of $(G,\infty)$ is zero for every $p>\frac{tr(\alpha)}{w_{k-1}}$ we will follow the idea of \cite[Proposition 10]{Pa08}. To that end it is necessary to see that the vertical flow $\varphi_t$ contracts exponentially the horizontal forms. 

We define another left-invariant norm on $G$: For every $v\in\R^n$ we write
\begin{equation}\label{descvect}
v=\sum_{i,j} v_{ij} +a\mathsmaller{\frac{\partial}{\partial t}},
\end{equation}
where the first sum corresponds to decomposition (\ref{descomposicion}). Given a point $(x,t)\in G$ we define
$$\langle v\rangle_{(x,t)}=\sum_{i,j}\|v_{ij}\|_{(x,t)}+|a|.$$
Since the subspaces $V_{ij}$ are invariant by $e^{t\alpha}$, we can easily see that the norm $\langle \text{ }\rangle_{(x,t)}$ is left-invariant and, as a consequence, equivalent to the Riemannian norm $\|\text{ }\|_{(x,t)}$. This gives us the following lemma:

\begin{lemma}\label{equiv}
Let $\omega$ be a $k$-form on $G$, then
$$|\omega|_{(x,t)}\asymp\sup\left\{|\omega_{(x,t)}(v_1,\ldots,v_k)| : \langle v_i\rangle_{(x,t)}=1\text{ for all }i=1,\ldots,k\right\},$$
where the constant does not depend on $\omega$ or the point $(x,t)\in G$.
\end{lemma}

A set of vectors in $\R^{n-1}$ is said to be $\alpha$-\textit{linearly independent} (denoted also by $\alpha$-LI) if it can be extended to a basis of the form $\bigcup_{i,j} \mathcal{B}_{ij}$,
where $\mathcal{B}_{ij}$ is a basis of $V_{ij}$.\\

\begin{lemma}\label{lemasup}
If $\omega$ is a horizontal $k$-form, then the supremum in Lemma $\ref{equiv}$ is reached on some $\alpha$-LI set.
\end{lemma}

In the previous lemma we can think of $\omega_{(x,t)}$ as an alternating $k$-linear map on $\R^{n-1}$.

\begin{proof}
Since the spheres for the norm $\langle\text{ }\rangle_{(x,t)}$ are compact, the supremum is reached on a set of vectors $v_1,\ldots,v_k\in\R^{n-1}$, with $\langle v_\ell\rangle_{(x,t)}=1$ for every $\ell=1,\ldots,k$. We write these vectors as in (\ref{descvect}):
$$v_\ell=\sum (v_\ell)_{ij}.$$
Then
\begin{align*}
\left|\omega_{(x,t)}(v_1,\ldots,v_k)\right|
    &=\left|\sum_{i,j}\omega_{(x,t)}((v_1)_{ij},v_2,\ldots,v_k)\right|\\
    &\leq \sum_{i,j}\|(v_1)_{ij}\|_{(x,t)}\left|\omega_{(x,t)}\left(\frac{(v_1)_{ij}}{\|(v_1)_{ij}\|_{(x,t)}},v_2,\ldots,v_k\right)\right|.
\end{align*}

Since $\langle v_1\rangle_{(x,t)}=\sum_{i,j}\|(v_1)_{ij}\|_{(x,t)}=1$, there exists a pair $(i_1,j_1)$ such that 
\begin{equation}\label{nueve}
|\omega_{(x,t)}(v_1,\ldots,v_k)|\leq\left|\omega_{(x,t)}\left(\frac{(v_1)_{i_1j_1}}{\|(v_1)_{i_1j_1}\|_{(x,t)}},v_2,\ldots,v_k\right)\right|.
\end{equation}
Observe that the vector
$u_1=\frac{(v_1)_{i_1j_1}}{\|(v_1)_{i_1j_1}\|_{(x,t)}}$ is unitary with respect to the norm $\langle\text{ }\rangle_{(x,t)}$ and belongs to $V_{i_1j_1}$. This implies that the inequality (\ref{nueve}) is in fact an equality. Continuing in this way we can construct an $\alpha$-LI set $\{u_1,\ldots,u_k\}$ that satisfies what we wanted.
\end{proof}

\medskip

%\begin{lemma}\label{lemitapolinomio} Let $P(s)=\sum_{i=0}^n{a_{2i}s^{2i}}$ be a polynomial with $a_0,a_{2n}>0$ and $a_{2i}\geq 0$ for all $i=1,\ldots,n-1$. Then there exists a polynomial $Q$ with $deg(Q)=2n$ such that for all $s\in \R$ and $t\geq 0$ we have
%$$\frac{P(s)}{P(s+t)}\geq Q(t).$$
%\end{lemma}

%\begin{proof} We distinguish three cases:
%\begin{itemize}
%\item If $s\geq 0$, then $\frac{P(s)}{P(s+t)}\leq 1$ because $P$ is increasing in $[0,+\infty)$.

%\item If $s<-2t$ we have that $P(s+t)\geq P(\frac{s}{2})$ because $P$ is decreasing in $(-\infty,0]$. Then
%$$\frac{P(s)}{P(s+t)}\leq \frac{P(s)}{P(\frac{s}{s})}\leq 2^{2n}.$$

%\item Finally, in the case $s\in [-2t,0]$, using again that $P$ is decreasing in $(-\infty,0]$, we have $P(s)\leq P(-2t)=P(2t)$, what implies
%$$\frac{P(s)}{P(s+t)}\leq \frac{P(s)}{a_0}\leq \frac{P(2t)}{a_0}\leq \frac{2^{2n}}{a_0}P(t).$$
%\end{itemize}
%We can take $Q(t)=\frac{2^{n}}{a_0}P(t)$.
%\end{proof}

%Note that the polynomials given by Lemma \ref{lemaformasgeneral}, part (i), satisfy the hypotesis of Lemma \ref{lemitapolinomio}.

\begin{lemma}\label{estimacionNormaVectores}
If $v\in V_{ij}$, there exists a positive polynomial $P_{ij}$ such that for every $(x,s)\in G$ and $t\geq 0$ we have
$$\|v\|_{(x,s+t)}\leq e^{-t\mu_i}\sqrt{P_{ij}(t)}\|v\|_{(x,s)}.$$
\end{lemma}

\begin{proof}

Observe that for every $s\in\R$  we have $\|v\|_{(x,s)}=\|e^{-s\alpha} v\|_0=e^{-s\mu_i}\|e^{-sJ} v\|_0$, where $J$ is the $(m_{ij}\times m_{ij})$-matrix 
$$J=J(m_{ij})=\left(\begin{array}{cccc} 0   & 1     &       &   \\
                                            &\ddots &\ddots &  \\
                                            &       &\ddots & 1   \\
                                            &       &       & 0 \end{array}\right).$$
Therefore,
$$\|v\|_{(x,s+t)}=e^{-(s+t)\mu_i}\|e^{-tJ}(e^{-sJ} v)\|_0\leq e^{-(s+t)\mu_i}\left|e^{-tJ}\right|\|e^{-sJ} v\|_0=e^{-t\mu_i}\left|e^{-tJ}\right|\|v\|_{(x,s)}.$$
Here $\left|e^{-tJ}\right|$ denotes the operator norm of the matrix $e^{-tJ}$. Since all norms on $\R^{m_{ij}^2}$ are biLipschitz equivalent, there exists a constant $C_{ij}>0$, depending only on $m_{ij}$, such that
$$\left|e^{-tJ}\right|\leq C_{ij}\sqrt{\sum_{1\leq \ell,r \leq m_{ij}} a_{\ell,r}(t)^2},$$
where $a_{\ell,r}$ are the entries of $e^{-tJ}$. Notice that they are polynomials in $t$, in particular $a_{\ell,\ell}=1$ for every $\ell=1,\ldots,m_{ij}$, then the Lemma follows by taking
$$P_{ij}(t)=C_{ij}^2\sum_{1\leq \ell,r \leq m_{ij}} a_{\ell,r}(t)^2.$$
\end{proof}

\begin{lemma}\label{contraccionflujo}
If $\omega$ is a horizontal $k$-form on $G$, then there exists a positive polynomial $Q$ such that for every $t\geq 0$,
$$|\varphi_t^*\omega|_{(x,s)}\preceq e^{-tw_k}\sqrt{Q(t)}|\omega|_{(x,s+t)}.$$
\end{lemma}

\begin{proof}
Using Lemmas \ref{equiv} and \ref{lemasup} we have
\begin{align*}
&|\varphi_t^*\omega|_{(x,t)} \asymp \max\left\{\left|\varphi_t^*\omega_{(x,s)}\left(\frac{v_1}{\|v_1\|_{(x,s)}},\ldots,\frac{v_k}{\|v_k\|_{(x,s)}}\right)\right| : \{v_1,\ldots,v_k\}\text{ is }\alpha\text{-LI}\right\}\\
    &=\max\left\{\prod_{\ell=1}^k\frac{\|v_\ell\|_{(x,s+t)}}{\|v_\ell\|_{(x,s)}}\left|\omega_{(x,s+t)}\left(\frac{v_1}{\|v_1\|_{(x,s+t)}},\ldots,\frac{v_k}{\|v_k\|_{(x,s+t)}}\right)\right| : \{v_1,\ldots,v_k\}\text{ is }\alpha\text{-LI}\right\}
\end{align*}
Suppose that $v_\ell\in V_{i_\ell j_\ell}$ for every $\ell=1,\ldots,k$, then by Lemma \ref{estimacionNormaVectores} and the fact that we are considering $\alpha$-LI sets we obtain
$$|\varphi_t^*\omega|_{(x,t)}\preceq e^{-tw_k}\sqrt{Q(t)}|\omega|_{(x,s+t)},$$
where $Q=\prod_{ij}(P_{ij})^k$.
\end{proof}

\begin{proposition}\label{proposicionLpHeintzeAnulacion} Let $k=2,\ldots,n$, then $L^p H^k(G,\infty)=0$ for all $p>\frac{\mathrm{tr}(\alpha)}{w_{k-1}}$.
\end{proposition}

\begin{proof} 
Take $\omega$ a closed form in $L^p\Omega^k(G,\infty)$. We want to construct an $L^p$-integrable differential $(k-1)$-form  $\vartheta$ such that $d\vartheta=\omega$ and $\vartheta =0$ on some neighborhood of $\infty$. By Remark \ref{suavizar} this will imply $L^pH^k(G,\infty)=0$.

Set
\begin{equation}\label{primitiva}
\vartheta=-\int_0^{+\infty}\varphi_t^*\iota_V\omega\ dt.    
\end{equation}
Observe that, since $\omega$ vanishes on a neighborhood of $\infty$, the above integral converges pointwise, thus $\vartheta$ is well-defined as a $k$-form. Furthermore, it is clear that $\vartheta$ is zero on some neighborhood of $\infty$. 

Since $\iota_V\omega$ is a horizontal form, by Lemma \ref{contraccionflujo} we have that for all $(x,s)\in G$ and $t\geq 0$,
$$|\varphi_t^*\iota_V\omega|_{(x,s)}\leq e^{-tw_k}\sqrt{Q(t)}|\iota_V\omega|_{(x,s+t)},$$
for some positive polynomial $Q$. Thus
\begin{align*}
\|\varphi_t^*\iota_V\omega\|_{L^p}^p
\leq \int_G e^{-t(pw_k-\mathrm{tr}(\alpha))} \sqrt{Q(t)} |\iota_V\omega|^p_{(x,s+t)}e^{-(s+t\mathrm{tr}(\alpha))}dxds=e^{-t\epsilon}\sqrt{Q(t)}\|\iota_V\omega\|^p_{L^p}
\end{align*}
where $\epsilon=pw_k-\mathrm{tr}(\alpha)>0$. It is easy to see that $|\iota_V\omega|_{(x,s)}\leq |\omega|_{(x,s)}$ for every $(x,s)\in G$, so $\|\varphi_t^*\iota_V\omega\|_{L^p}\leq Ce^{-t\epsilon}\|\omega\|_{L^p}$. This implies that the integral (\ref{primitiva}) converges in $L^p(M,\Lambda^{k-1})$. We have to prove that it is smooth and $d\vartheta=\omega$.

We know that there exists $T\in\R$ such that $\iota_V\omega_{(x,s)}=0$ for all $s\geq T$, then $\vartheta_{(x,s)}$ is an integral on a compact interval for every $(x,s)\in M$. Since $(x,s,t)\mapsto \varphi_t^*\iota_V\omega$ is smooth we can use Lemma \ref{Leibniz} to see that $\vartheta$ is in $\Omega^{k-1}(M)$ and 
$$d\vartheta=-\int_0^{+\infty}d(\varphi_t^*\iota_V\omega) dt$$

The Lie derivative of $\omega$ with respect to the vertical field $V$ is
$L_V\omega=\left.\frac{d}{dt}\right|_{t=0}\varphi_t^*\omega$.
Observe that $\frac{d}{dt}\varphi_t^*\omega=\varphi_t^*L_V\omega$. Thus, using the Cartan formula $L_V\omega=d\iota_V\omega+\iota_V d\omega$ (see for example \cite[Chapter I,Section A]{GHL}) and that $\omega$ is closed, we obtain
\begin{align*}
\varphi^*_t\omega-\omega
    &= \int_0^t \frac{d}{ds}\varphi_s^*\omega\ ds
    = \int_0^t \varphi_s^*(d\iota_V\omega+\iota_Vd\omega)ds
    = \int_0^t d(\varphi_s^*\iota_V\omega)ds.
\end{align*}
For every $(x,r)\in G$ we have
$$\omega_{(x,r)}=\lim_{t\to+\infty}\left(\varphi^*_t\omega_{(x,r)}-\int_0^t d(\varphi_s^*\iota_V\omega)_{(x,r)}ds\right).$$
The limit exists because the expression in brackets is constant for $t$ big enough. Then we conclude 
$$\omega_{(x,r)}=-\int_0^{+\infty} d(\varphi_s^*\iota_V\omega)_{(x,r)}ds =d\vartheta_{(x,r)}$$
for all $(x,t)\in G$, which finishes the proof.
\end{proof}

We prove the second part of Theorem \ref{TeoremaLpHeintzeCalculo} by studying two cases separately. 

\begin{proposition}\label{proposicionLpHeintzeNoAnulacion} Let $k=1,\ldots,n-1$, then $L^p H^k(G,\infty)\neq 0$ for
$\frac{\mathrm{tr}(\alpha)}{w_k}<p<\frac{\mathrm{tr}(\alpha)}{w_{k-1}}$. In the case $k=1$ we read $\frac{\mathrm{tr}(\alpha)}{w_0}=+\infty$.
\end{proposition}

\begin{proof}
We want to construct a closed differential $k$-form $\omega$ on $G$ which represents a non-zero class in $L^p H^k(G,\infty)$. We work again with the complex $(L^p\Omega^*(G,\infty),d)$. The strategy of this proof is inspired by the duality ideas mentioned above; that is: we will give a $(n-k)$-form $\beta\in L_{\mathrm{loc}}^{q}\Omega^{n-k}(G,\infty)$, with $\frac{1}{p}+\frac{1}{q}=1$, such that
\begin{enumerate}
    \item[(j)] $\nu_\beta(\omega)=\int_G\omega\wedge\beta\neq 0$, and
    \item[(k)] $dL^p\Omega^{k-1}(G,\infty)\subset\Ker \nu_\beta;$
\end{enumerate}
which shows that $\omega$ represents a non-zero element in $L^p H^k(G,\infty)$.

Consider two smooth functions $f:\R^{n-1}\to [0,1]$ and $g:\R\to [0,1]$ such that $\supp(f)$ is compact, $g(t)=0$ for all $t\geq 1$ and $g(t)=1$ for all $t\leq 0$. Then define
\begin{equation}\label{FormaNoNula}
\omega_{(x,t)}=d\left(f(x)g(t)\ dx_1\wedge\ldots\wedge dx_{k-1}\right).    
\end{equation}
Using triangular inequality we have
$$\|\omega\|_{L^p}\leq \|f g'\ dt\wedge dx_1\wedge\cdots\wedge dx_{k-1}\|_{L^p}+\sum^{n-1}_{j=k}\left\|\frac{\partial f}{\partial x_j} g\ dx_j\wedge dx_1\wedge\cdots\wedge dx_{k-1}\right\|_{L^p}.$$

The first term is finite because $fg'$ is smooth and has compact support. Then it is enough to show that for every $j=k,...,n-1$ the form $\omega_j=\frac{\partial f}{\partial x_j} g\ dx_j\wedge dx_1\wedge\cdots\wedge dx_{k-1}$ belongs to $L^p$. By Lemma \ref{lemaformasgeneral} there exists a positive polynomial $P_j$ such that 
$$\|\omega_j\|_{L^p}^p\preceq \left\|\frac{\partial f}{\partial x_j}\right\|_{L^p}^p \int^1_{-\infty}e^{t(p(w_{k-1}+\lambda_j)-\mathrm{tr}(\alpha))}P_j(t)^{\frac{p}{2}}dt.$$
Hence $\|\omega_j\|_{L^p}<+\infty$ if $p>\frac{\mathrm{tr}(\alpha)}{w_{k-1}+\lambda_j}$ for every $j=k,\ldots,n-1$, which implies $\|\omega\|_{L^p}<+\infty$ for every $p>\frac{\mathrm{tr}(\alpha)}{w_k}$.

Define $\beta=dx_k\wedge\ldots\wedge dx_{n-1}$. To prove that $\beta$ is in $L_{\mathrm{loc}}^q\Omega^{n-k}(G,\infty)$ it is enough to show that it is $q$-integrable on $Z=B_R\times (-\infty,T)$ for every ball  $B_R=B_R(0,R)\subset\R^{n-1}$ and $T\in\R$. By Lemma \ref{lemaformasgeneral} there exists a positive polynomial $P$ such that  
\begin{align*}
\|\beta\|_{L^q,Z}^q   
%    \preceq \int^T_{-\infty}\int_{B_R} e^{qt(\lambda_k+\cdots+\lambda_{n-1})}P(t)^{\frac{q}{2}}e^{-t\mathrm{tr}(\alpha)}dxdt
    \leq\Vol(B_R)\int^T_{-\infty}e^{t(q(\lambda_k+\cdots+\lambda_{n-1})-\mathrm{tr}(\alpha))}P(t)^{\frac{q}{2}}dt.
\end{align*}
This integral converges if and only if $q>\frac{\mathrm{tr}(\alpha)}{\lambda_k+\cdots+\lambda_{n-1}}$, or equivalently $p<\frac{\mathrm{tr}(\alpha)}{w_{k-1}}$.

We now prove (j): Let $B_{R_1}\in\R^{n-1}$ be a ball such that $\supp(f)\subset B_{R_1}$. For $t<1$ consider $Z_t=B_{R_1}\times [t,1]$. Since $|\omega\wedge\beta|$ is in $L^1(G)$ because of Hölder's inequality, we have by Stokes' theorem:  
\begin{align*}
\int_G \omega\wedge\beta 	&=\lim_{t\to-\infty}\int_{Z_t} d(fg\ dx_1\wedge\cdots\wedge dx_{n-1})
%	&=\lim_{t\to-\infty}\int_{B_{R_1}\times \{t\}} fg\ dx_1\wedge\cdots\wedge dx_{n-1}\\
	=\int_{B_{R_1}}f\ dx_1\wedge\cdots\wedge dx_{n-1}\neq 0.
\end{align*}

In order to prove (k) we take $\vartheta\in L^p\Omega^{k-1}(G,\infty)$. There exist two constant $R_2,T_2>0$ such that the support of $\vartheta$ is contained in $B_{R_2}\times (-\infty,T_2]$. By Stokes' theorem
$$\nu_\beta(d\vartheta)=\int_G d\vartheta\wedge\beta = \lim_{t\to-\infty} \int_{B_{R_2}\times [t,T_2]}d\vartheta\wedge\beta= \lim_{t\to-\infty}\int_{B_{R_2}\times \{t\}} \vartheta\wedge\beta.$$
In the second equality we use again that $|d\vartheta\wedge\beta|$ is in $L^1(G)$. Suppose that $\nu_\beta(d\vartheta)\neq 0$, then there exist $\epsilon>0$ and $t_0$ such that for all $t\leq t_0,$
\begin{equation}\label{referencia1}
\left|\int_{B_{R_2}\times \{t\}} \vartheta\wedge\beta\right| >\epsilon.
\end{equation}

For $I=(i_1,\ldots,i_{k-1},j_1,\ldots,j_{k-1},\ell_1,\ldots,\ell_{k-1})\in\mathcal{I}_{k-1}$ we consider
\begin{align*}
(\tilde{\upsilon}_I)_{(x,t)} &=(L_{(x,t)}^{-1})^*dx_I=(L_{(x,t)}^{-1})^*dx_{i_1j_1}^{\ell_1}\wedge\cdots\wedge (L_{(x,t)}^{-1})^*dx_{i_{k-1}j_{k-1}}^{\ell_{k-1}}\\
    &=e^{-tw_I}\left(\sum_{h=0}^{M_1}\frac{(-t)^h}{h!}dx_{i_1j_1}^{\ell_i+h}\right)\wedge\cdots\wedge\left(\sum_{h=0}^{M_{k-1}}\frac{(-t)^h}{h!}dx_{i_{k-1}j_{k-1}}^{\ell_{k-1}+h}\right),
\end{align*}
where $M_s=m_{i_sj_s}-\ell_s$. We define $(\upsilon_I)_{(x,t)}=e^{tw_I}(\tilde{\upsilon}_I)_{(x,t)}$ and write
$\vartheta=\sum_{I\in\mathcal{I}_{k-1}}a_I \upsilon_I$.
Observe that $|\upsilon_I|_{(x,t)}\asymp e^{tw_I}$ for every $(x,t)\in G$.

Since $\{\upsilon_I:I\in\mathcal{I}_{k-1}\}$ is orthogonal at every point with respect to $\langle\langle\text{ , }\rangle\rangle_{(x,t)}$, then $[\vartheta]_{(x,t)}\geq [a_I \upsilon_I]_{(x,t)}$ for every $I\in\Delta_{k-1}$ and as a consequence $|\vartheta|_{(x,t)}\succeq |a_I \upsilon_I|_{(x,t)}$.

We can easily observe that 
$$\int_{B_{R_2}\times\{t\}}\vartheta\wedge\beta=\int_{B_{R_2}\times\{t\}}a_{I_0} dx_1\wedge\cdots\wedge dx_{n-1},$$ 
where $I_0$ is such that $dx_{I_0}=dx_1\wedge\ldots\wedge dx_{k-1}$. Hence, inequality \eqref{referencia1} means that there exist $\epsilon>0$ and $t_0$ such that for every $t\leq t_0$,
$$\left|\int_{B_R\times\{t\}}a_{I_0}(x,t) dx \right|>\epsilon.$$
Now we have
\begin{align*}
\|\vartheta\|_{L^p}^p
    &\succeq \int_G |a_{I_0} \upsilon_{I_0}|^p_{(x,t)}dV_{(x,t)}\\
    &\succeq\int_{-\infty}^{t_0}\left(\int_{B_R}|a_{I_0}(x,t)|^p dx\right)e^{t(pw_{k-1}-\mathrm{tr}(\alpha))}dt\\
    &\succeq \epsilon^p\int_{-\infty}^{t_0}e^{t(pw_{k-1}-\mathrm{tr}(\alpha))}dt =+\infty.
\end{align*}
This contradiction proves that $\nu_\beta(d\vartheta)=0$.
\end{proof}

\begin{proposition}\label{proposicionLpHeintzeCritico}
If $p=\frac{\mathrm{tr}(\alpha)}{w_{k-1}}$ with $k=2,\ldots,n-1$, then $L^p H^k(G,\infty)\neq 0$. 
\end{proposition}

\begin{proof}
We consider $\omega$ and $\beta$ as in the proof of Proposition \ref{proposicionLpHeintzeNoAnulacion}. The main difficulty to apply the previous argument in this case is that $\beta$ does not belong to $L^q_{\mathrm{loc}}\Omega^{n-k}(G,\infty)$, then $\nu_\beta$ is not well-defined. An alternative is to consider the function $$\tilde{\nu}_\beta:L^p\Omega^k(G,\infty)\to [0,+\infty],\ \tilde{\nu}_\beta(\varpi)=\liminf_{t\to-\infty}\left|\int_{\R^{n-1}\times [t,+\infty)}\varpi\wedge\beta\right|,$$
which is well-defined because $\supp(\varpi)\cap(\R^{n-1}\times [t,+\infty))$ is compact for every $t\in\R$.

It is clear that
$$\tilde{\nu}_\beta(\omega)=\int_{\R^{n-1}}f(x)\ dx\neq 0.$$
Furthermore we can show using the above argument that $\tilde{\nu}_\beta(d\vartheta)=0$ for all $\vartheta\in L^p\Omega^{k-1}(G,\infty)$. This implies that $\omega$ represents a non-zero class in the relative $L^p$-cohomology of $(G,\infty)$.
\end{proof}

Theorem \ref{TeoremaLpHeintzeCalculo} is obtained from Propositions \ref{proposicionLpHeintzeAnulacion}, \ref{proposicionLpHeintzeNoAnulacion} and \ref{proposicionLpHeintzeCritico}.

\section{Relative $L^p$-cohomology in degree one and Besov spaces}

There exists a direct relation between $\ell^p H^1(X)$ and $\ell^p H^1(X,\xi)$ for a Gromov-hyperbolic simplicial complex $X$ with bounded geometry and $\xi\in\partial X$. Indeed, if $\theta$ is a closed $1$-cochain in $\ell^p (X^{(1)},\xi)$, then it is zero in $\ell^p H^1(X,\xi)$ if and only if it is zero in $\ell^p H^1 (X)$. This is because if $f$ is a function in $\ell^p (X^{(0)})$ such that $\delta f=\theta$, then $f$ must be constant (and hence zero) in a neighborhood of $\xi$. Thus, there is a canonical injection $\ell^p H^1(X,\xi)\hookrightarrow \ell^p H^1(X)$. Combining this fact with Theorem \ref{equivalencia} and \cite[Theorem 3]{GKS88} we obtain the first part of Theorem \ref{TeoGr1}.

By \cite{BP}, if $X$ is a Gromov-hyperbolic simplicial complex with bounded geometry such that there exists an Ahlfors regular metric $d$ on the conformal gauge of the visual boundary $\partial X$, then for every $k\in\N$ and $p\geq 1$ the cohomology space $\ell^pH^1(X)$ is isomorphic to the Besov space
$$B_p(\partial X)=\bigl\{u:\partial X : \|u\|_{B_p}<+\infty\bigr\}/\R,$$
where 
$$\|u\|_{B_p}=\left(\int_{Z\times Z} \frac{|u(\xi)-u(\eta)|^p}{d(\xi,\eta)^{2Q}} d\mathcal{H}(\xi)d\mathcal{H}(\eta)\right)^{1/p},$$
and $\R$ indicates the spaces of almost everywhere constant functions. Here $Q$ is the Hausdorff dimension of $(\partial X,d)$ and $\mathcal{H}$ is the corresponding Hausdorff measure.

Remember that $(\partial X,d)$ is \textit{Ahlfors regular} of dimension $Q$ if there exists a constant $K$ such that for every $\xi\in \partial X$ and $r>0$,
$$K^{-1}r^{Q}\leq \mathcal{H}(B(\xi,r))\leq K r^{Q}.$$
And a metric $d$ is in the \textit{conformal gauge} of the visual boundary $\partial X$ if it is cuasi-symmetric equivalent to any visual metric. See \cite[Chapter 15]{H} for more details. 

Combining this identification with the previous observation we can identify the space $\ell^pH^1(X,\xi)$ (where $\xi\in \partial X$) with
$$B_p(\partial X,\xi)=\bigl\{u:\partial X\to\R : \|u\|_{B_p}<+\infty\text{ and }u\equiv\text{cte on a neighborhood of }\xi\bigr\}/\R.$$

Sometimes it is convenient to consider the Besov algebra
$$A_p(\partial X)=\bigl\{u:\partial X\to\R : u\text{ is continuous and }\|u\|_{B_p}<+\infty\bigr\},$$
which is a Banach algebra with the norm 
$\|\ \|=\|\ \|_{\infty}+\|\ \|_{B_p}$,
and its maximal ideals are
$$\mathcal{I}_\xi=\bigl\{u\in A_p(\partial X): u(\xi)=0\bigr\}.$$
As before, we can consider the relative Besov algebra
$$A_p(\partial X,\xi)=\bigl\{u\in A_p(\partial X) : u\equiv 0\text{ on a neighbourhood of }\xi\bigr\}\subset \mathcal{I}_\xi.$$
All these definitions can be done for general compact metric spaces (not only for boundaries of Gromov-hyperbolic spaces). 

We will prove the following result:

\begin{theorem}\label{densidadBesov}
Let $(Z,d)$ be a compact  Ahlfors regular  metric space of dimension $Q>0$  and $p> Q$. Suppose that $z_0\in Z$ satisfies that there exists $R_0>0$ such that for every $R\in (0,R_0]$ the sets $\overline{B(z_0,R)}$ and $B(z_0,R)^c$ are continua. Then
\begin{enumerate}
    \item[(i)] $B_p(Z,z_0)$ is dense in $B_p(Z)$.
    \item[(ii)] $A_p(Z,z_0)$ is dense in $\mathcal{I}_{z_0}$.
\end{enumerate}
\end{theorem}

If $X$ is a Gromov-hyperbolic simplicial complex with bounded geometry such that the conformal gauge of $\partial X$ has a metric satisfying the conditions of Theorem \ref{densidadBesov}, then $\ell^p H^1(X,\xi)$ is dense in $\ell^p H^1(X)$ for every $\xi\in\partial X$ and $p>Q$. The same result is true, because of Theroem \ref{equivalencia}, in the case of Riemannian manifolds with bounded geometry. This proves the second part of Theorem \ref{TeoGr1}.

Let $C$ and $D$ be two disjoint non-degenerated continua in a metric space $(Z,d)$. Their \textit{relative distance} is
$$\Delta(C,D)=\frac{dist(C,D)}{\min\{\mathrm{diam}(C),\mathrm{diam}(D)\}},$$
and their \textit{Besov capacity} (as defined in \cite{B06}) is
\begin{equation}\label{Capacitancia}
\Omega_p(C,D)=\inf\left\{\|u\|_{B_p}^p : u\in A_p(Z), u|_{C}\leq 0\text{ and }u|_{D}\geq 1\right\}.    
\end{equation}

Observe that if $u\in A_p(Z)$ is a function such that $u|_C\leq 0$ and $u|_D\geq 1$, then we can take $\tilde{u}$ defined by $\tilde{u}(z)=0$ if $u(z)\leq 0$, $\tilde{u}(z)=1$ if $u(z)\geq 1$ and $\tilde{u}(z)=u(z)$ otherwise. Since $|\tilde{u}(x)-\tilde{u}(y)|\leq |u(x)-u(y)|$ for every $x,y\in Z$, we have $\|\tilde{u}\|_{B_p}\leq \|u\|_{B_p}$. Hence, in \eqref{Capacitancia} we can take the infimum among functions taking values in $[0,1]$.

The following theorem will be key to prove Theorem \ref{densidadBesov}:

\begin{theorem}[\cite{B06}]\label{TeoMarc}
Let $(Z,d)$ be a compact Ahlfors regular metric space of dimension $Q$ and $p>Q$. Then there exist two decreasing homeomorphisms $\varphi,\psi:(0,+\infty)\to (0,+\infty)$ such that for two non-degenerated disjoing continua $C,D\subset Z$,
$$\varphi\bigl(\Delta(C,D)\bigr)\leq \Omega_p(C,D)\leq\psi\bigl(\Delta(C,D)\bigr).$$
\end{theorem}

If $[u]\in B_p(Z)$, the function 
$$F_u(x,y)=\frac{|u(x)-u(y)|^p}{d(x,y)^{2Q}}$$ 
belongs to $L^1(\mathcal{H}\times \mathcal{H})$, thus
we can consider on $Z\times Z$ a measure given by $d\nu_u=F_u d\mathcal{H}d\mathcal{H}$. We will use this measure in the following lemma and in the proof of Theorem \ref{densidadBesov}.

\begin{lemma}\label{lemaBesovAcotadas}
Let $Z$ be as above, thus the set
$$B_p^\infty(Z)=\bigl\{[u]\in B_p(Z) : u\in L^\infty(Z)\bigr\}$$
is dense in $B_p(Z)$.
\end{lemma} 

\begin{proof}
Let $[u]\in B_p(Z)$. For every $n\in\N$ we consider
$$X_n=\{z\in Z : |u(z)|>n\}.$$
The sequence $\{X_n\}$ is decreasing, $\Hh(X_n)<+\infty$ and $\bigcap_n X_n=\emptyset$. We define
$$u_n(z)=\left\{\begin{array}{cc} u(z) & \text{ if }z\notin X_n \\ \frac{u(z)n}{|u(z)|}  & \text{ if }z\in X_n \end{array}\right.$$
and $v_n=u-u_n$. 

Observe that $u_n\in L^\infty(Z)$, $|v_n(x)-v_n(y)|\leq |u(x)-u(y)|$ for every $x,y\in Z$ and $|v_n(x)-v_n(y)|=0$ if $x,y\notin X_n$. Therefore, decomposing 
$$Z\times Z=(Z\times X_n)\cup (X_n\times Z)\cup (X_n^c\times X_n^c)$$ 
and using symmetry we have
\begin{align*}
\|v_n\|_{B_p}^p &\leq  2\int_{Z\times X_n} \frac{|v_n(x)-v_n(y)|^p}{d(x,y)^{2Q}}\,d\mathcal{H}(x)d\mathcal{H}(y)\\
&\leq 2\int_{Z\times X_n} \frac{|u(x)-u(y)|^p}{d(x,y)^{2Q}}\,d\mathcal{H}(x)d\mathcal{H}(y)=2\nu_u(Z\times X_n)\to 0. 
\end{align*}
\end{proof}

\begin{proof}[Proof of Theorem \ref{densidadBesov}]
By Lemma \ref{lemaBesovAcotadas}, to prove $(i)$ we need to show that every element $[u]\in B^\infty_p(Z)$ can be approximated by elements of $B_p(Z,z_0)$. 

We assume that $R_0$ also satisfies $$\diam \bigl(B(z_0,R_0)^c\bigr)\geq \diam \bigl(B(z_0,R_0)\bigr)\geq R_0.$$ 
For every $R\in (0,R_0]$ we denote $B_R=B(z_0,R)$, then for every $r\in (0,R)$,
$$\Delta\bigl(B_R^c,\overline{B_r}\bigr)\leq \frac{R}{r}$$
Using Theorem \ref{TeoMarc} we have
$$\Omega_p\bigl(B_R^c,\overline{B_r}\bigr)\leq \psi\left(\frac{R}{r}\right)\to 0, \text{ when }r\to 0.$$
Hence, for every $R\in (0,R_0]$ we can take $v_R\in A_p(Z,z_0)$ such that $\|v_R\|_{B_p}\leq R$, $v_R(z)\in [0,1]$ for every $z\in Z$ and $v_R(z)=1$ for every $z\notin B_R$. Then we consider $u_R(z)=u(z)v_R(z)$, let us prove that $\|u-u_R\|_{B_p}\to 0$ when $R\to 0$.
\begin{align*}
\|u-u_R\|^p_{B_p}   &\leq 2\int_{Z\times B_R}\frac{|(1-v_R(x))u(x)-(1-v_R(y))u(y)|^p}{d(x,y)^{2Q}}\,d\mathcal{H}(x)d\mathcal{H}(y) \\
& \preceq  \int_{Z\times B_R} |u(x)|^p\frac{|(1-v_R(x))-(1-v_R(y))|^p}{d(x,y)^{2Q}} \,d\mathcal{H}(x)d\mathcal{H}(y)\\
&+\int_{Z\times B_R}  |1-v_R(y)|^p\frac{|u(x)-u(y)|^p}{d(x,y)^{2Q}}\, d\mathcal{H}(x)d\mathcal{H}(y)\\
&\leq \|u\|^p_\infty\|v_R\|_{B_p}^p+\|1-v_R\|_\infty^p\nu_u(Z\times B_R).
\end{align*}
In the second line we added and subtracted $(1-v_R(y))u(x)$ and then applied Jensen's inequality. Since $u$ and $(1-v_R)$ are bounded and $\nu_u(Z\times \{z_0\})=0$ (because $Z$ is Ahlfors regular), the expression in the last line converges to $0$ when $R\to 0$, which proves $(i)$.

If $u\in\mathcal{I}_{z_0}$, then $u_R\in A_p(Z,z_0)$ for every $R\in (0,R_0]$. The previous calculation shows that $\|u-u_R\|_{B_p}\to 0$ when $R\to 0$. In addition we have
$$\|u-u_R\|_\infty\leq \sup\bigl\{|u(z)| : z\in B_R\bigr\},$$
which converges to $0$ when $R\to 0$ because $u$ is continuous and $u(z_0)=0$. This proves $(ii)$.
\end{proof}

\section*{Acknowledgments} 

Most of this work is a part of my thesis at \textit{Universidad de la República} and \textit{Université de Lille}, supported by both institutions. I am deeply grateful to my advisors Marc Bourdon and Matías Carrasco for shearing ideas with me and helping me correct the manuscript. I also thank Yves Cornulier for some helpful suggestions.

%\bibliographystyle{alpha}
%\bibliography{ref.bib}

\begin{thebibliography}{11}

\bibitem{A}
\textsc{O. Attie.}
\newblock{\em Quasi-isometry classification of some manifolds of bounded geometry.}
Math. Z., {\bf 216}:501--527, 1994.

\bibitem{B06}
\textsc{M. Bourdon.}
\newblock{\em An algebraic characterization of quasi-möbius homeomorphisms.}
Ann. Acad. Sci. Fenn. Math., {\bf 32}:235--250, 2007.

\bibitem{BP}
\textsc{M. Bourdon and H. Pajot.}
\newblock{\em Cohomologie $l_p$ et espaces de Besov.}
J. Reine Angew Math., {\bf 558}:85--108, 2003.

\bibitem{C}
\textsc{M. Carrasco Piaggio.}
\newblock{\em Orlicz spaces and the large scale geometry of Heintze groups.}
Math. Ann., {\bf 368}:433--481, 2017.

\bibitem{CS}
\textsc{M. Carrasco Piaggio and E. Sequeira.}
\newblock{\em On quasi-isometry invariants associated to a Heintze group.}
Geom. Dedicata, {\bf 189}(1):1--16, 2017.

\bibitem{Cor}
\textsc{Y. Cornulier.}
On   the   quasi-isometric   classification   of   locally   compactgroups. (New Directions in Locally Compact Groups, 275--342.)
\emph{Cambridge University Press}, 2018.

\bibitem{GHL}
\textsc{S. Gallot, D. Hulin and J. Lafontaine.}
Riemannian Geometry.
\newblock{\em Sringer-Verlag}, 1989.

\bibitem{G}
\textsc{L. Genton.}
\newblock{\em Scaled Alexander-Spanier cohomology and Lqp cohomology for metric spaces.}
Thesis 2014.

\bibitem{GH}
\textsc{E. Ghys and P. de la Harpe.}
Sur les groupes hyperboliques d'après Mikhael Gromov.
\newblock{\em Progress in Mathematics, vol 83. Birkhäuser Boston}
, 1990.

\bibitem{GKS2}
\textsc{V.M. Gol'dshtein, V.I. Kuz'minov and I.A. Shvedov.}
\newblock{\em Dual spaces of spaces of differential forms.}
Sib. Math. J., {\bf 27}(1):35--44, 1986.

\bibitem{GKS88}
\textsc{V.M. Gol'dshtein, V.I. Kuz'minov and I.A. Shvedov.}
\newblock{\em The de Rham isomorphism of $l_p$-cohomology of non compact Riemannian manifold.}
Sib. Math. J., {\bf 29}(2):190--197, 1988.

\bibitem{GT98}
\textsc{V.M. Gol'dshtein, M. Troyanov}
\newblock{\em The $L_{q,p}$-cohomology of SOL.}
Ann. Fac. Sci. Touluse, {\bf 7}:687--698, 1998.

\bibitem{GT06}
\textsc{V.M. Gol'dshtein, M. Troyanov}
\newblock{\em Soboler inequalities for differential forms and $L_{q,p}$-cohomolgy.}
J. Geo. Anal., {\bf 16}(4):597--631, 2016.

\bibitem{GT10}
\textsc{V.M. Gol'dshtein, M. Troyanov}
\newblock{\em A short proof of the Hölder-Poincaré duality for $l_p$-cohomology.}
Rend. Semin. mat. R. Univ. Padova., {\bf 124}(1):179--184, 2010.

\bibitem{H}
\textsc{J. Heinonen.}
\newblock{\em Lectures on analysis on metric spaces.}
Universitext. 
Springer-Verlag, New York, 2001.

\bibitem{Heintze74}
\textsc{E. Heintze.}
\newblock{\em On homogeneous manifolds of negative curvature.}
Math. Ann., {\bf 211}:23--34, 1974.

\bibitem{LX}
\textsc{E. Le Donne and X. Xie.}
\newblock{\em Rigidity of fiber-preserving quasisymmetric maps.}
Rev. Mat. Iberoam., {\bf 32}(4):1407--1422, 2016.

\bibitem{MT}
\textsc{J. Mackay and J. Tyson}
Conformal Dimension: Theory and Application, \emph{(University Lecture Series; Vol. 54) American Mathematical Society, 2010.}

\bibitem{Pa89a}
\textsc{P. Pansu.}
\newblock{\em Métriques de Carnot-Carathéodory et quasiisométries des espaces symétriques de rang un.}
Ann. of Math. (2),  {\bf 129}(1):1--60, 1989.

\bibitem{Pa95}
\textsc{P. Pansu.}
\newblock{\em Cohomologie $L^p$: invariance sour quasiisométries.}
Preprint: https://www.math.u-psud.fr/~pansu/liste-prepub.html, 1995.

\bibitem{Pa08}
\textsc{P. Pansu.}
\newblock{\em Cohomologie $L^p$ et piencement.}
Comment. Math. Helv., {\bf 83}(2):327--357, 2008.

\bibitem{W}
\textsc{H. Whitney.}
Geometric integration theory.
\newblock{\em Dover Publications, Inc.,} 1957.

\bibitem{Xie14}
\textsc{X. Xie.}
\newblock{\em Large scale geometry of negative curved $\R^n\rtimes\R$.}
Geom. Topol., {\bf 18}(2):831--872, 2014.


\end{thebibliography}

\medskip
\medskip

\address

\end{document}